\documentclass[a4paper,10pt]{article}
\usepackage[utf8]{inputenc}
\usepackage{amsmath, amsthm, amsfonts, amssymb}
\usepackage[usenames,dvipsnames]{color}
\usepackage{fullpage}
\usepackage{graphicx}
\usepackage{authblk}

\usepackage[matrix,arrow,cmtip,rotate,curve,arc]{xy}

\newcommand{\Diagonal}{\Delta}

\newcommand\R{\ensuremath{\mathbb{R}}}
\newcommand\N{\ensuremath{\mathbb{N}}}
\newcommand\Z{\ensuremath{\mathbb{Z}}}
\newcommand\Q{\ensuremath{\mathbb{Q}}}
\newcommand{\RP}{\ensuremath{\mathbb{RP}}}
\newcommand{\Rspace}{\mathbf{R}}

\newcommand\F{\mathcal{F}}
\newcommand\G{\mathcal{G}}
\renewcommand\P{\mathcal{P}}
\DeclareMathOperator\sd{sd}

\DeclareMathOperator\ima{Im}

\newcommand\eps{\ensuremath{\varepsilon}}

\newcommand{\pth}[1]{\left( #1 \right)}

\newcommand{\delprod}[1]{\widetilde{#1}}
\newcommand{\s}{\mathbb{S}}
\newcommand{\obs}[1]{\mathfrak{o}^{#1}}

\DeclareMathOperator{\conv}{conv}

\newcommand{\U}[1]{U_{\overline{#1}}}
\newcommand{\nontrivial}{nontrivial}

\newcommand{\simplex}[1]{\Delta_{#1}}
\newcommand{\skel}[2]{#2^{(#1)}}
\newcommand{\skelsim}[2]{\skel{#1}{\simplex{#2}}}

\theoremstyle{plain}
\newtheorem{theorem}{Theorem}
\newtheorem{claim}{Claim}
\newtheorem{lemma}[theorem]{Lemma}
\newtheorem{proposition}[theorem]{Proposition}
\newtheorem{corollary}[theorem]{Corollary}
\newtheorem{observation}[theorem]{Observation}
\theoremstyle{definition}
\newtheorem{definition}[theorem]{Definition}
\theoremstyle{remark}
\newtheorem{example}[theorem]{Example}
\newtheorem{remark}[theorem]{Remark}

\newcounter{sideremark}

\title{Bounding Helly numbers via Betti numbers\footnote{2010
    Mathematics Subject Classification: 52A35 (Helly-type theorems and
    geometric transversal theory) and 55U10 (Simplicial sets and
    complexes). Keywords: Helly-type theorems, Ramsey's theorem,
    Embeddings of simplicial complexes, Betti numbers.}~\footnote{An
    extended abstract of this work was presented at the \emph{31st
      International Symposium on Computational
      Geometry}~\cite{BH-Socg}.}~\footnote{PP, ZP and MT were
    partially supported by the Charles University Grant GAUK 421511.
    ZP was partially supported by the Charles University Grant
    SVV-2014-260103. ZP and MT were partially supported by the ERC
    Advanced Grant No.~267165 and by the project CE-ITI (GACR
    P202/12/G061) of the Czech Science Foundation. UW was partially
    supported by the Swiss National Science Foundation (grants
    SNSF-200020-138230 and SNSF-PP00P2-138948).  Part of this work was
    done when XG was affiliated with INRIA Nancy Grand-Est and when MT
    was affiliated with Institutionen f\"{o}r matematik, Kungliga
    Tekniska H\"{o}gskolan, then IST Austria.}}  \author[1]{Xavier
  Goaoc} \author[2]{Pavel Pat\'ak} \author[3]{Zuzana Pat\'akov\'a}
\author[3]{Martin Tancer} \author[4]{Uli Wagner}

\affil[1]{\small Universit\'e Paris-Est Marne-la-Vall\'ee, France.}
\affil[2]{\small Department of Algebra, Charles University, Prague, Czech Republic.}
\affil[3]{\small Department of Applied Mathematics, Charles University, Prague, Czech Republic.}
\affil[4]{IST Austria, Klosterneuburg, Austria.}

\begin{document}
\maketitle

\hfill
\begin{minipage}{8cm}
  \begin{flushright}
      \emph{ Dedicated to the memory of Ji\v{r}\'{\i} Matou\v{s}ek,
          wonderful teacher, mentor, collaborator, and~friend.}
  \end{flushright}
\end{minipage}

\begin{abstract}
  We show that very weak topological assumptions are enough to ensure
  the existence of a Helly-type theorem. More precisely, we show that
  for any non-negative integers $b$ and $d$ there exists an integer
  $h(b,d)$ such that the following holds. If $\F$ is a finite family
  of subsets of $\R^d$ such that $\tilde\beta_i\pth{\bigcap\G} \le b$
  for any $\G \subsetneq \F$ and every $0 \le i \le \lceil d/2
  \rceil-1$ then $\F$ has Helly number at most $h(b,d)$. Here
  $\tilde\beta_i$ denotes the reduced $\Z_2$-Betti numbers (with
  singular homology). These topological conditions are sharp: not
  controlling any of these $\lceil d/2 \rceil$ first Betti numbers
  allow for families with unbounded Helly number.

  Our proofs combine homological non-embeddability results with a
  Ramsey-based approach to build, given an arbitrary simplicial
  complex $K$, some well-behaved chain map $C_*(K) \to C_*(\R^d)$.
\end{abstract}

\section{Introduction}

Helly's classical theorem~\cite{h-umkkm-23} states that a finite
family of convex subsets of $\R^d$ must have a point in common if any
$d+1$ of the sets have a point in common.  Together with Radon's and
Caratheodory's theorems, two other ``very finite properties'' of
convexity, Helly's theorem is a pillar of combinatorial geometry.
Along with its variants (\emph{eg.} colorful or fractional), it
underlies many fundamental results in discrete geometry, from the
centerpoint theorem~\cite{rado} to the existence of weak
$\eps$-nets~\cite{epsnet} or the
$(p,q)$-theorem~\cite{piercing-number}.

In the contrapositive, Helly's theorem asserts that any finite family
of convex subsets of $\R^d$ with empty intersection contains a
sub-family of size at most $d+1$ that already has empty intersection.
This inspired the definition of the \emph{Helly number} of a family
$\F$ of arbitrary sets. If $\mathcal{F}$ has empty intersection then
its Helly number is defined as the size of the largest sub-family
$\mathcal{G} \subseteq \F$ with the following properties:
$\mathcal{G}$ has empty intersection and any proper sub-family of
$\mathcal{G}$ has nonempty intersection; if $\mathcal{F}$ has nonempty
intersection then its Helly number is, by convention, $1$.  With this
terminology, Helly's theorem simply states that any finite family of
convex sets in $\R^d$ has Helly number at most $d+1$. 

Helly already realized that bounds on Helly numbers independent of the
cardinality of the family are not a privilege of
convexity: his \emph{topological} theorem~\cite{h-usvam-30} asserts
that a finite family of open subsets of~$\R^d$ has Helly number at
most $d+1$ if the intersection of any sub-family of at most $d$
members of the family is either empty or a \emph{homology
  cell}.\footnote{By definition, a homology cell is a topological
  space $X$ all of whose (reduced, singular, integer coefficient)
  homology groups are trivial, as is the case if $X=\R^d$ or $X$ is a
  single point. Here and in what follows, we refer the reader to
  standard textbooks like
  \cite{Hatcher:AlgebraicTopology-2002,Munkres:AlgebraicTopology-1984}
  for further topological background and various topological notions
  that we leave undefined.} Such \emph{uniform} bounds are often
referred to as \emph{Helly-type theorems}. In discrete geometry,
Helly-type theorems were found in a variety of contexts, from simple
geometric assumptions (\emph{eg.} homothets of a planar convex
curve~\cite{Swanepoel}) to more complicated implicit conditions (sets
of line intersecting prescribed geometric
shapes~\cite{Tverberg-translates,Cremona-convexity,Transversals-spheres},
sets of norms making a given subset of $\R^d$
equilateral~\cite[Theorem~5]{Petty-Minkowski}, etc.) and several
surveys~\cite{e-hrctt-93, w-httgt-04, tancer13} were devoted to this
abundant literature. These Helly numbers give rise to similar
finiteness properties in other areas, for instance in variants of
Whitney's extension problem~\cite{whitnext} or the combinatorics of
generators of certain groups~\cite{farb}.

Many Helly numbers are established via ad hoc arguments, and decades
sometimes go by before a conjectured bound is effectively proven, as
illustrated by Tverberg's proof~\cite{Tverberg-translates} of a
conjecture of Gr\"unbaum~\cite{Grunbaum-squares}. This is true not
only for the quantitative question (\emph{what is the best bound?})
but also for the existential question (\emph{is the Helly number
  uniformly bounded?}); in this example, establishing a first
bound~\cite{Katchalski-translate} was already a matter of decades.
Substantial effort was devoted to identify general conditions ensuring
bounded Helly numbers, and \emph{topological conditions}, as opposed
to more geometric ones like convexity, received particular attention.
The general picture that emerges is that requiring that intersections
have \emph{trivial} low-dimensional homotopy~\cite{m-httucs-97} or
have \emph{trivial} high-dimensional homology~\cite{CdVGG} is
sufficient (see below for a more comprehensive account).

\subsection{Problem statement and results}

In this paper, we focus on the existential question and give the
following new homological sufficient condition for bounding Helly
numbers. Throughout the paper, we consider homology with coefficients\footnote{The choice of $\Z_2$ as the ring of 
coefficient ring has two reasons. On the one hand, we work with the van Kampen obstruction to prove certain non-embeddability results, and the obstruction is naturally defined either for integer coefficients or over $\Z_2$ (it is a torsion element of order two). On the other hand, the Ramsey arguments used in our proof require working over a fixed finite ring of coefficients to ensure a finite number of color classes (\emph{cf.} Claim~\ref{c:uniform-i}).}
in $\Z_2$, and denote by $\tilde\beta_i(X)$ the $i$th reduced Betti
number (over $\Z_2$) of a space~$X$. Furthermore, we use the notation
$\bigcap \mathcal{F}:=\bigcap_{U\in \F} U$ as a shorthand for the
intersection of a family of sets.

\begin{theorem}\label{t:main}
  For any non-negative integers $b$ and $d$ there exists an integer
  $h(b,d)$ such that the following holds. If $\F$ is a finite family
  of subsets of $\R^d$ such that $\tilde\beta_i\pth{\bigcap\G} \le
  b$ for any $\G \subsetneq \F$ and every $0 \le i \le \lceil d/2
  \rceil-1$ then $\F$ has Helly number at most $h(b,d)$.
\end{theorem}

Our proof, which we sketch in Subsection~\ref{ss:outline}, hinges
on a general principle, which we learned from
Matou\v{s}ek~\cite{m-httucs-97} but which already underlies the
classical proof of Helly's theorem from Radon's lemma, to derive
Helly-type theorems from results of non-embeddability of certain
simplicial complexes. 
%
%
The novelty of our approach is to examine these non-embeddability
arguments from a homological point of view. This turns out to be a
surprisingly effective idea, as homological analogues of embeddings
appear to be much richer and easier to build than their homotopic
counterparts. More precisely, our proof of Theorem~\ref{t:main} builds
on two contributions of independent interest:

\begin{itemize}
\item We reformulate some non-embeddability results in homological
  terms. We obtain a homological analogue of the Van Kampen-Flores
  Theorem (Corollary~\ref{c:nohomrep}) and, as a side-product, a
  homological version of Radon's lemma (Lemma~\ref{l:HomRadon}). This
  is part of a systematic effort to translate various homotopy
  technique to a more tractable homology setting. It builds on, and
  extends, previous work on homological
  minors~\cite{Wagner:MinorsRandomExpandingHypergraphs-2011}.
\item By working with homology rather than homotopy, we can generalize
  a technique of Matou\v{s}ek~\cite{m-httucs-97} that uses Ramsey's
  theorem to find embedded structures. In this step, roughly speaking,
  we construct some auxiliary (chain) map, with certain homological
  constraints, inductively by increasing the dimension of the preimage
  complex while decreasing the size of it. This approach turned out to
  be also useful in a rather different setting, regarding the
  (non-)embeddability of skeleta of complexes into manifolds
  ~\cite{kuhnel_SoCG15}.
\end{itemize}

\noindent
Our method also proves:
\begin{itemize}
\item A bound of $d+1$ on the Helly number of any family $\F$ of
  subsets of $\R^d$ such that $\tilde\beta_i\pth{\bigcap\G}=0$ for all
  $\G\subsetneq \F$ and all $i\leq d$ (see
  Corollary~\ref{c:helly_hom_radon}), which generalizes Helly's
  topological theorem as the sets of $\F$ are, for instance, not
  assumed to be open. (In the original proof, this assumption
    is crucial and used to ensure that the union of the sets must have
    trivial homology in dimensions larger than $d$; this may fail if
  the sets are not open.)
\item A bound of $d + 2$ on the Helly number of any family $\F$ of
  subsets of $\R^d$ such that $\tilde\beta_i\pth{\bigcap\G}=0$ for all
  $\G\subsetneq \F$ but only for $i\leq \lceil d/2\rceil - 1$ (see
  Corollary~\ref{c:helly_d/2-connected}).
\end{itemize}

\noindent
In both cases the bounds are tight.


Quantitatively, the bound on $h(b,d)$ that we obtain in the general
case is very large as it follows from successive applications of
Ramsey's theorems. The conditions of Theorem~\ref{t:main} relax the
conditions of a Helly-type theorem of Amenta~\cite{a-npihtt-96} (see
the discussion below) for which a lower bound of $b(d+1)$ is
known~\cite{Larman}; a stronger lower bound is possible
for $h(b,d)$ (see Example~\ref{ex:lowerbound}) but we consider 
narrowing this gap further to be outside the scope of the present paper.
Qualitatively, Theorem~\ref{t:main} is sharp in the sense that all
(reduced) Betti numbers $\tilde\beta_i$ with $0 \le i \le \lceil d/2
\rceil-1$ need to be bounded to obtain a bounded Helly number (see
Example~\ref{ex:necessary}).

\begin{example}\label{ex:lowerbound}
First, we observe that for every $d \geq 2$ there is a geometric simplicial
complex $\Gamma_d$ with $d + 2$ vertices, embedded in $\R^d$, such that every 
nonempty induced subcomplex $L$ of $\Gamma_d$ is connected and 
satisfies $\tilde\beta_i(L)=0$ for $i\neq d-1$ and $\tilde\beta_{d-1}(L) \leq 1$.


Indeed, we can take $\Gamma_d$ to be the stellar subdivision of the
$d$-simplex (i.e., the cone over the boundary of the $d$-simplex): Among the
vertices of $\Gamma_d$, $d+1$ of them, say $v_1, \dots, v_{d+1}$, form a $d$-simplex,
and the last one, say $w$, is situated in the barycenter of that simplex. The
maximal simplices of $\Gamma_d$ contain $w$ and $d$ of the vertices $v_i$.  Given an
induced subcomplex $L$, either $L$ misses one of the $v$-vertices, and then $L$
is a $k$-simplex for some $k\leq d$; or $L$ contains all the $v_i$, in which case either $L
= \Gamma_d$ or $L$ is the boundary of the simplex spanned by the vertices $v_i$.

Now, let $\Gamma_{b,d}$ be a complex that consists of $b$ disjoint copies of
$\Gamma_d$, embedded in $\R^d$.
For a vertex $v$ of $\Gamma_{b,d}$, let $U_v$ be the union of all simplices of
$\Gamma_{b,d}$ not containing $v$
(i.e., $U_v$ is the geometric realization of the induced subcomplex of
$\Gamma_{b,d}$ on all vertices but $v$). We define $\F$ to be the collection of all subcomplexes $F_v$,
where $v$ ranges over all vertices of $\Gamma_{b,d}$. Thus, by construction, $\F$ contains $b(d+2)$ sets,
$\bigcap \F = \emptyset$, and for any nonempty proper subsystem $\G \subset \F$, the intersection 
$\bigcap \G$ is nonempty, and by the properties of $\Gamma_d$, the reduced Betti numbers of $\bigcap \G$ are bounded by $b$.%
\footnote{We remark that this construction can be further improved (at the
cost of simplicity). For example, for $d = 3$, it is possible to find a
geometric simplicial complex $\Gamma'_3$ with six vertices (instead of five) with
properties analogous to $\Gamma_3$: Consider a simplex $\Delta \subseteq \R^3$ with 
vertices $v_1, v_2, v_3$ and $v_4$. Let $b$ the barycenter of this simplex and 
we set $v_5$ to be the barycenter of the triangle $v_1v_2b$ and $v_6$ to be 
the barycenter of $v_3v_4b$. Finally, we set $\Gamma'_3$ to be the subdivision of
$\Delta$ with vertices $v_1, \dots, v_6$ and with maximal simplices $1245$,
$1235$, $3416$, $3426$, $5613$, $5614$, $5623$, and $5624$ where the label $ABCD$ 
stands for $\conv\{v_A,v_B,v_C,v_D\}$. One can check that this indeed yields a simplicial
complex with the required properties. See the $1$-skeleton of
$\Gamma'_3$ in~Figure~\ref{f:linked_triangles}. We believe that an analogous example
can be also constructed for $d \geq 4$.}

\begin{figure}
  \begin{center}
    \includegraphics{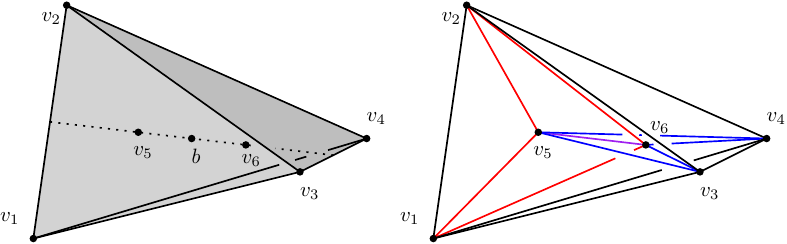}
    \caption{The simplex $\Delta$ (left) and the $1$-skeleton of $\Gamma'_3$ (right).}
    \label{f:linked_triangles}
  \end{center}
\end{figure}

\end{example}
 
\begin{example}\label{ex:necessary}
  Let us fix some $k$ with $0\leq k\leq \lceil d/2 \rceil-1$.  For $n$
  arbitrarily large, consider a geometric realization in $\R^d$ of the
  $k$-skeleton of the $(n-1)$-dimensional simplex (see
  \cite[Section~1.6]{Matousek:BorsukUlam-2003}); more specifically,
  let $V=\{v_1,\ldots,v_n\}$ be a set of points in general position in
  $\R^d$ 
  and consider
  all geometric simplices $\sigma_A:=\conv(A)$ spanned by subsets
  $A\subseteq V$ of cardinality $|A|\leq k+1$.  

  Similarly as in the previous example, let $U_j$ be the union of 
  all the simplices not containing the vertex $v_j$, 
  for $1 \leq j \leq n$. We set $\F = \{U_1, \dots,U_n\}$. Then,
  $\bigcap \F=\emptyset$, and for any proper sub-family
  $\mathcal{G}\subsetneq \mathcal{F}$, the intersection $\bigcap \G$
  is either~$\R^d$ (if $\G=\emptyset$) or (homeomorphic to) the
  $k$-dimensional skeleton of a $(n-1-|\G|)$-dimensional simplex.
  Thus, the Helly number of $\mathcal{F}$ equals $n$.  Moreover, the
  $k$-skeleton $\skelsim{k}{m-1}$ of an $(m-1)$-dimensional simplex
  has reduced Betti numbers $\tilde\beta_i=0$ for $i\neq k$ and
  $\tilde\beta_k=\binom{m-1}{k+1}$. Thus, we can indeed obtain
  arbitrarily large Helly number as soon as at least one
  $\tilde\beta_k$ is unbounded.
\end{example}

\subsection{Relation to previous work}

The search for topological conditions that ensure bounded Helly
numbers started with Helly's topological theorem~\cite{h-usvam-30}
(see also~\cite{d-httdb-70} for a modern version of the proof) and
organized along several directions related to classical questions in
topology. Theorem~\ref{t:main} unifies topological conditions
originating from two different approaches:
\begin{itemize}

\item Helly-type theorem can be derived from non-embeddability
  results, in the spirit of the classical proof of Helly's theorem
  from Radon's lemma. Using this approach,
  Matou\v{s}ek~\cite{m-httucs-97} showed that it is sufficient to
  control the \emph{low-dimensional homotopy} of intersections of
  sub-families to ensure bounded Helly numbers: for any non-negative
  integers $b$ and $d$ there exists a constant $c(b,d)$ such that any
  finite family of subsets of $\R^d$ in which every sub-family
  intersects in at most $b$ connected components, each \emph{$\pth{
  \lceil d/2 \rceil-1}$-connected}, has Helly number at most $c(b,d)$. 
  (We recall that a
    topological space $X$ is $k$-connected, for some integer $k\geq
    0$, if every continuous map $S^i\to X$ from the $i$-dimensional
    sphere to $X$, $0\leq i\leq k$, can be extended to a map
  $D^{i+1}\to X$ from the $(i+1)$-dimensional disk to $X$.) 
    By Hurewicz' Theorem and the
  Universal Coefficient Theorem \cite[Theorem~4.37 and
  Corollary~3A.6]{Hatcher:AlgebraicTopology-2002}, a $k$-connected
  space $X$ satisfies $\tilde\beta_i(X)=0$ for all $i\leq k$. Thus,
  our condition indeed relaxes Matou\v{s}ek's, in two ways: by using
  $\Z_2$-homology instead of the homotopy-theoretic assumptions of
  $k$-connectedness\footnote{We also remark that our condition can be
    verified algorithmically since Betti numbers are easily
    computable, at least for sufficiently nice spaces that can be
    represented by finite simplicial complexes, say. By contrast, it
    is algorithmically undecidable whether a given $2$-dimensional
    simplicial complex is $1$-connected, see, e.g., the survey
    \cite{Soare:ComputabilityDifferentialGeometry-2004}.}, and by
  allowing an arbitrary fixed bound $b$ instead of $b=0$.



\item Helly's topological theorem can be easily derived from classical
  results in algebraic topology relating the homology/homotopy of the
  nerve of a family to that of its union: Leray's \emph{acyclic cover
    theorem}~\cite[Sections III.4.13, VI.4 and VI.13]{Br-sheaf-97} for
  homology, and Borsuk's \emph{Nerve
    theorem}~\cite{b-iscsc-48,b-nfhg-03} for homotopy (in that case
  one considers finite open \emph{good covers}\footnote{An open good
    cover is a finite family of open subsets of $\R^d$ such that the
    intersection of any sub-family is either empty or is contractible
    (and hence, in particular, a homology cell).}). More general Helly
  numbers were obtained via this approach by Dugundji~\cite{Dugundji},
  Amenta~\cite{a-npihtt-96}\footnote{The role of nerves is implicit in
    Amenta's proof but becomes apparent when compared to an earlier
    work of Wegner~\cite{Wegner} that uses similar ideas.}, Kalai and
  Meshulam~\cite{km-lnpthtt-08}, and\footnote{The result of Colin de
    Verdi\`ere et al.~\cite{CdVGG} holds in any paracompact
    topological space; Theorem~\ref{t:main} only subsumes the $\R^d$
    case.}  Colin de Verdi\`ere et al.~\cite{CdVGG}. The outcome is
  that if a family of subsets of $\R^d$ is such that any sub-family
  intersects in at most $b$ connected components, each a homology cell
  (over $\Q$), then it has Helly number at most $b(d+1)$. This
  therefore relaxes Helly's original assumption by allowing
  intersections of sub-families to have $\tilde{\beta}_0$'s bounded by
  an arbitrary fixed bound $b$ instead of $b=0$.  Theorem~\ref{t:main}
  makes the same relaxation for the $\tilde{\beta}_1$'s,
  $\tilde{\beta}_2$'s, $\ldots \tilde{\beta}_{\lceil d/2 \rceil-1}$'s
  and drops \emph{all} assumptions on higher-dimensionnal homology,
  including the requirement that sets be open (which is used to
  control the $(>d)$-dimensional homology of intersections).
\end{itemize}

\noindent
Let us highlight two Helly-type results that stand out in this line of
research as \emph{not} subsumed (qualitatively) by
Theorem~\ref{t:main}. On the one hand, Eckhoff and Nischke~\cite{EN}
gave a purely combinatorial argument that derives the theorems of
Amenta~\cite{a-npihtt-96} and Kalai and Meshulam~\cite{km-lnpthtt-08}
from Helly's convex and topological theorems. On the other hand,
Montejano~\cite{Montejano-Berge13} relaxed Helly's original assumption
on the intersection of sub-families of size $k \le d+1$ from being a
homology cell into having trivial $d-k$ homology (so only one Betti
number needs to be controlled for each intersection, but it must be
zero). These results neither contain nor are contained in
Theorem~\ref{t:main}.


We remark that another non-topological structural condition, known to
ensure bounded Helly numbers, also falls under the umbrella of
Theorem~\ref{t:main}. As observed by Motzkin~\cite[Theorem~7]{Motzkin}
(see also Deza and Frankl~\cite{DezaFrankl-Algebraic}), any family of
real algebraic subvarieties of $\R^d$ defined by polynomials of degree
at most $k$ has Helly number bounded by a function of $d$ and $k$
(more precisely, by the dimension of the vector subspace of
$\R[x_1,x_2, \ldots, x_d]$ spanned by these polynomials); since the
Betti numbers of an algebraic variety in $\R^n$ can be bounded in
terms of the degree of the polynomials that define
it~\cite{Milnor,Thom}, this also follows from Theorem~\ref{t:main}.
We give some other examples in Section~\ref{s:further}, where we
easily derive from Theorem~\ref{t:main} generalizations of various
existing Helly-type theorems.
%

\bigskip

Note that Theorem~\ref{t:main} is similar, in spirit, to some of the
general relations between the growth of Betti numbers and
\emph{fractional} Helly theorems conjectured by Kalai and
Meshulam~\cite[Conjectures~6 and~7]{kalai_conjectures}.  Kalai and
Meshulam, in their conjectures, allow a polynomial growth of the Betti
numbers in $|\bigcap \G|$. As the following example shows,
Theorem~\ref{t:main} is also sharp in the sense that even a linear
growth of Betti numbers, already in $\R^1$, may yield unbounded Helly
numbers. In particular, the conjectures of Kalai and Meshulam cannot be
strengthened to include Theorem~\ref{t:main}.

\begin{example}
  Consider a positive integer $n$ and open intervals $I_i := (i - 1.1;
  i + 0.1)$ for $i \in [n]$. Let $X_i := [0,n] \setminus I_i$.  The
  intersection of all $X_i$ is empty but the intersection of any
  proper subfamily is nonempty. In addition, the intersection of $k$
  such $X_i$ can be obtained from $[0,n]$ by removing at most $k$ open
  intervals, thus the reduced Betti numbers of such an intersection are
  bounded by $k$.
\end{example}

\subsection{Further consequences}\label{s:further}

We conclude this introduction with a few implications of
our main result.

\paragraph{New geometric Helly-type theorems.}

The main strength of our result is that very weak topological assumptions
on families of sets are enough to guarantee a bounded Helly number.
This can be used to identify new Helly-type theorems, for instance by easily
detecting generalizations of known results, as we now
illustrate on two Helly-type theorems of Swanepoel.

A first example is given by a Helly-type theorem for hollow
boxes~\cite{Swanepoel-hollow}, which generalizes (qualitatively) as
follows:


%
%


\begin{corollary}
\label{cor:generalized-swanepoel}
  For all integers $s,d\geq 1$, there exists an integer $h'(s,d)$ such
  that the following holds. Let $S$ be a set of $s$ nonzero vectors in $\R^d$,
  and let $\F=\{U_1,U_2, \ldots, U_n\}$ where each $U_i$ is
  a polyhedral subcomplex of some polytope $P_i$ in $\R^d$ which can be obtained as an
  intersection of half-spaces with normal vectors 
  in $S$. Then $\F$ has Helly number at most $h'(s,d)$.
\end{corollary}

\noindent
Swanepoel's result corresponds to the case $S = \{\pm e_1,
\pm e_2, \ldots, \pm e_d\}$ where $e_1,\ldots, e_d$ form a basis
of $\R^d$.

\begin{proof}[Proof of Corollary~\ref{cor:generalized-swanepoel}]
  We verify the assumptions of Theorem~\ref{t:main}, i.e., we
  consider a subfamily $\G = \{U_i\colon i \in I\} \subseteq \F$ and we check that
  $\tilde\beta_i(\bigcap \G)$ is bounded by a function of $s$ and $d$ for any $i \geq 0$
  (to apply Theorem~\ref{t:main}, 
  it would be sufficient to consider $i \leq \lceil d/2 \rceil - 1$, but in the present setting, there is no
  difference in reasoning for other values of $i$).  

  Let $\P = \P(S)$ be the set of all polytopes which can be obtained as an
    intersection of half-spaces with normal vectors 
    in $S$. Let $P_i \in \P$ be a polytope such that $U_i$ is a polyhedral subcomplex of
  $P_i$. 
  
  Let us consider the polytope $P = \bigcap_{i \in I} P_i$. From the definition of
  $\P$ we immediately deduce that $P \in \P$. Moreover, the intersection $U :=
  \bigcap \G$ is a polyhedral subcomplex of $P$. (The faces $U$ are of form
  $\bigcap_{i \in I} \sigma_i$ where $\sigma_i$ is a face of $U_i$;
  see~\cite[Exercise~2.8(5) + hint]{rourke-sanderson72}.)


  Since $P \in \P$ we deduce that it has at most $2s$ facets. By the
  dual version of the upper bound theorem~\cite[Theorem~8.23]{ziegler95}, the number of faces of $P$ is bounded by a function of $s$
  and $d$. Consequently, $\tilde\beta_i(U)$ is bounded by a function of $s$
    and $d$, since $U$ is a subcomplex of $P$.
\end{proof}

A second example concerns a Helly-type theorem for families of
translates and homothets of a convex curve~\cite{Swanepoel}, which are
special cases of families of \emph{pseudo-circles}. More generally, a
family of \emph{pseudo-spheres} is defined as a set $\F=\{U_1, U_2,
\ldots, U_n\}$ of subsets of $\R^d$ such that or any $\G \subseteq
\F$, the intersection $\cap(\G)$ is homeomorphic to a $k$-dimensional
sphere for some $k \in \{0,1,\ldots, d-1\}$ or to a single point.
The case $b=1$ of Theorem~\ref{t:main} immediately implies the following:

\begin{corollary}\label{c:pseudo}
  For any integer $d$ there exists an integer $h(d)$ such that the
  Helly number of any finite family of pseudo-spheres in $\R^d$ is at
  most $h(d)$.
\end{corollary}

\noindent
We note that the special case of Euclidean spheres falls under the umbrella of intersections 
of real algebraic varieties of bounded degree, for which the Helly number is bounded as observed by Motzkin and others, as discussed above~\cite{Maehara,DezaFrankl-Algebraic}. 
For the more general setting pseudo-spheres, however, the above result is new, to the best of our knowledge. 
An optimal bound $h(d) = d + 1$ as soon as the family contains at least $d+3$
pseudo-spheres was obtained by Sosnovec~\cite{sosnovec15}, after discussing the
contents of Corollary~\ref{c:pseudo} with us.


\paragraph{Generalized linear programming.}

Theorem~\ref{t:main} also has consequences in the direction of
optimization problems. Various optimization problems can be formulated
as the minimization of some function $f: \R^d \to \R$ over some
intersection $\bigcap_{i=1}^n C_i$ of subsets $C_1,C_2, \ldots, C_n$
of $\R^d$. If, for $t \in \R$, we let $L_t = f^{-1}\pth{(-\infty,t]}$
and $\F_t = \{C_1, C_2, \ldots, C_n, L_t\}$ then
\[ \min_{x \in \bigcap_{i=1}^n C_i} f(x) = \min\left\{t \in \R:
  \bigcap \F_t \neq \emptyset\right\}.\]
If the Helly number of the families $\F_t$ can be bounded
\emph{uniformly} in $t$ by some constant $h$ then there exists a
subset of $h-1$ constraints $C_{i_1}, C_{i_2}, \ldots, C_{i_{h-1}}$
that suffice to define the minimum of $f$:
\[ \min_{x \in \bigcap_{i=1}^n C_i} f(x) = \min_{x \in
  \bigcap_{j=1}^{h-1} C_{i_j}} f(x).\]
A consequence of this observation, noted by Amenta~\cite{a-httgl-94},
is that the minimum of $f$ over $C_1 \cap C_2 \cap \ldots \cap C_n$
can\footnote{This requires $f$ and $C_1, C_2, \ldots, C_n$ to be
  generic in the sense that the number of minima of $f$ over $\cap_{i
    \in I} C_i$ is bounded uniformly for $I\subseteq \{1,2,\ldots,
  n\}$.}  be computed in randomized $O(n)$ time by \emph{generalized
  linear programming}~\cite{sw-cblprp-92} (see de Loera et al.~\cite{deLoera_RS}
for other uses of this idea). Together with Theorem~\ref{t:main}, this
implies that an optimization problem of the above form can be solved
in randomized linear time if it has the property that every
intersection of some subset of the constraints with a level set of the
function has bounded ``topological complexity'' (measured in terms of
the sum of the first $\lceil d/2 \rceil$ Betti numbers). Let us
emphasize that this linear-time bound holds in a real-RAM model of
computation, where any constant-size subproblems can be solved in
$O(1)$-time; it therefore concerns the \emph{combinatorial difficulty}
of the problem and says nothing about its \emph{numerical difficulty}.

\subsection{Proof outline}
\label{ss:outline}
  Let us briefly sketch the proof of Theorem~\ref{t:main}.

  \bigskip
  
  Consider the simplified setting where we have
  subsets $A_1, A_2, \ldots, A_5$ of $\R^2$ such that any
  four have non-empty intersection and any three have path-connected
  intersection. Draw $K_5$, the complete graph on $5$ vertices,
  \emph{inside} the union of the five sets by picking points $p_i \in
  \cap_{j \neq i} A_j$ and connecting any two points $p_u$, $p_v$
  inside the intersection $\cap_{j \neq u,v} A_j$. The (stronger form
  of the) non-planarity of $K_5$ ensures that two edges that share no
  vertex must cross, and the intersection point witnesses that
  $\cap_{i=1}^5 A_i$ is non-empty (\emph{cf.} Figure~\ref{f:k5}). This
  idea, more systematically, ensures that any family of planar sets
  with path-connected intersections has Helly number at most $4$.
  
  Now consider subsets $A_1, A_2, \ldots, A_n$ of $\R^2$ such that the
  intersection of any proper subfamily is nonempty and has at most $b$
  path-connected components. We can again pick $p_i \in \cap_{j \neq
    i} A_j$. Two points $p_u$, $p_v$ may end up in different connected
  components of $\cap_{j \neq u,v} A_j$, but among any $b+1$ points
  $p_{i_1}, p_{i_2}, \ldots, p_{i_{b+1}}$, two can be connected inside
  $\cap_{j \neq i_1, i_2, \ldots, i_{b+1}} A_j$. We can thus still
  draw a large graph inside the union, but each edge misses an extra
  set of $A_i$'s. A Ramsey-type argument ensures that for $n$ large
  enough, we can find a copy of $K_5$ where each edge misses distinct
  extra sets, and therefore that $\cap_{i=1}^n A_i$ is non-empty. 

  These arguments generalize to higher dimension: once we can draw
  $p_up_v$, $p_vp_w$ and $p_up_w$ inside the intersection of some
  family of subsets, we can fill the triangle in that intersection if
  it is $1$-connected (in homotopy). More systematically, given a
  family of subsets of $\R^{2k}$ whose proper intersections are
  $k$-connected (in homotopy), we can draw $\skelsim{k}{2k+2}$ inside
  their union and find, via the Van Kampen-Flores theorem, that the
  complete intersection is non-empty (and similarly in odd
  dimensions). This is, in short, Matou\v{s}ek's
  theorem~\cite{m-httucs-97}. 

  \bigskip

  \noindent
  \begin{minipage}{11cm}
    \quad We extend Matou\v{s}ek's approach to allow intersections to
    have bounded but non-trivial homotopy in dimension $1$ or
    more. The main difficulty is that we may not be able to fill any
    elementary cycle: as illustrated on the right-hand figure, for $n$
    arbitrarily large, $K_n$ can be drawn in an annulus so that no
    triangle can be filled. There still exist cycles that can be
    filled, for instance $2435$; they are simply not boundaries of
    triangles. Such cycles are more easily found by working with the
    additive structure of $\Z_2$-homology: the sum of any two
    homologous cycles is a boundary (and therefore ``fillable''), and
    many pairs of homologous cycles exist because a bounded Betti
    number ensures a constant number of homology classes.
\end{minipage}
\hfill  
\begin{minipage}{5cm}
  \begin{center}\includegraphics[width=3.5cm,keepaspectratio]{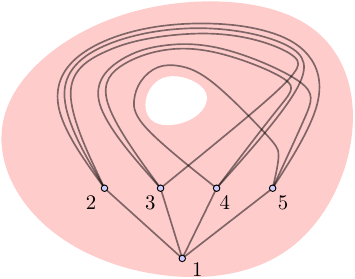}\end{center}
\end{minipage}

\smallskip

The key idea is, then, to look for sufficiently large sets of vertices
where, as in the example above, every triangle has the same
$\Z_2$-homology, and to map the \emph{barycentric subdivision} of a
triangle to these vertices (as described in Figure~\ref{f:alpha}); the
resulting sum of evenly many homologous simplices must be a
boundary. These large sets of vertices with homologous triangles exist
as soon as the Betti number is bounded: indeed, one can simply apply
Ramsey's theorem to the $3$-uniform hypergraph on the vertices where
every triangle is ``colored'' by its homology class. This idea
  generalizes to arbitrary dimension.

Because of the switch to homology, we do not build a map of
$\skelsim{k}{2k+2}$ into the target space $\R^d$ ($d = 2k$ or $d =
2k-1$) but only a chain map from the simplicial chain complex of
$\skelsim{k}{2k+2}$ into the singular chain complex of $\R^d$.  Hence,
we can no longer rely on the classical non-embeddability results and
have to develop homological analogs.

\bigskip


We set up our homological machinery in
Section~\ref{sec:homological-ae} (homological almost-embeddings,
homological Van Kampen-Flores Theorem, and homological Radon
lemma). We then spell out, in Section~\ref{s:history},
variations of the technique that derives Helly-type theorems from
non-embeddability. We finally introduce our refinement of this
technique and the proof of Theorem~\ref{t:main} in
Section~\ref{s:ccm}.

\subsection{Notation}


We assume that the reader is familiar with basic topological notions
and facts concerning simplicial complexes and singular and simplicial
homology, as described in textbooks like
\cite{Hatcher:AlgebraicTopology-2002,Munkres:AlgebraicTopology-1984}.
As remarked above, throughout this paper we will work with homology
with $\Z_2$-coefficients unless explicitly stated otherwise.
Moreover, while we will consider singular homology groups for
topological spaces in general, for simplicial complexes we will work
with simplicial homology groups. In particular, if $X$ is a
topological space then $C_\ast(X)$ will denote the singular chain
complex of $X$, while if $K$ is a simplicial complex, then $C_\ast(K)$
will denote the simplicial chain complex of $K$ (both with
$\Z_2$-coefficients).

\bigskip

We use the following notation. Let $K$ be a (finite, abstract)
simplicial complex. The \emph{underlying topological space} of $K$ is
denoted by $|K|$. Moreover, we denote by $K^{(i)}$ the
\emph{$i$-dimensional skeleton} of $K$, i.e., the set of simplices of
$K$ of dimension at most $i$; in particular $K^{(0)}$ is the set of
vertices of $K$.  For an integer $n\geq 0$, let $\simplex{n}$ denote
the $n$-dimensional simplex.



\subsection*{Acknowledgments}

We would like to express our immense gratitude to Ji\v{r}{\'\i}
Matou\v{s}ek, not only for raising the problem addressed in the
present paper and valuable discussions about it, but, much more
generally, for the privilege of having known him, as our teacher,
mentor, collaborator, and friend. Through his tremendous depth and
insight, and the generosity with which he shared them, he greatly
influenced all of us.

We further thank J\"{u}rgen Eckhoff for helpful comments on a
preliminary version of the paper, and Andreas Holmsen and Gil Kalai for
providing us with useful references.


\section{Homological Almost-Embeddings}
\label{sec:homological-ae}

In this section, we define \emph{homological almost-embedding}, an
analogue of topological embeddings on the level of chain maps, and
show that certain simplicial complexes do not admit homological
almost-embeddings in $\R^d$, in analogy to classical non-embeddability
results due to Van Kampen and Flores. In fact, when this comes at no
additional cost we phrase the auxiliary results in a slightly more
general setting, replacing $\R^d$ by a general topological space
$\Rspace$. Readers that focus on the proof of Theorem~\ref{t:main} can
safely replace every occurrence of $\Rspace$ with $\R^d$.

\subsection{Non-Embeddable Complexes}

We recall that an \emph{embedding} of a finite simplicial complex $K$ into
  $\R^d$ is simply an injective continuous map $|K|\to \R^d$. The fact
  that the complete graph on five vertices cannot be embedded in the
  plane has the following generalization.

\begin{proposition}[Van
  Kampen~\cite{vanKampen:KomplexeInEuklidischenRaeumen-1932},
  Flores~\cite{Flores:NichtEinbettbar-1933}]
  \label{prop:vK-F} For $k\geq 0$, the complex $\skelsim{k}{2k+2}$, the
  $k$-dimensional skeleton of the $(2k+2)$-dimensional simplex, cannot
  be embedded in $\R^{2k}$.
\end{proposition}

A basic tool for proving the non-embeddability of a simplicial complex is the so-called \emph{Van Kampen obstruction}.
To be more precise, we emphasize that in keeping with our general convention regarding coefficients, we work with the $\Z_2$-coefficient version\footnote{There is also a version of the Van Kampen obstruction with integer coefficients, which in general yields more precise information regarding embeddability than the $\Z_2$-version, but we will not need this here. We refer to  \cite{Melikhov:vanKampenRelatives-2009} for further background.} of the Van Kampen obstruction, which will be reviewed in some detail
in Section~\ref{sec:obstructions} below.
Here, for the benefit of readers who are willing to accept certain topological
facts as given, we simply collect those statements necessary to motivate the
definition of homological almost-embeddings and to follow the logic of the proof of Theorem~\ref{t:main}.

Given a simplicial complex $K$, one can define, for each $d\geq 0$, a certain
cohomology class $\obs{d}(K)$ that resides in the cohomology group
$H^d(\overline{K})$ of a certain auxiliary complex $\overline{K}$ (the quotient
of the combinatorial deleted product by the natural $\Z_2$-action, see below);
see the paragraph on obstructions following
Lemma~\ref{lem:homological-Gauss-map} for a more proper definition of $\obs
d(K)$.
This cohomology class $\obs{d}(K)$ is called the Van Kampen obstruction to
embeddability into $\R^d$ because of the following fact:

\begin{proposition}
\label{prop:obstruction-almost-embeddings}
Suppose that $K$ is a finite simplicial complex with $\obs{d}(K)\neq 0$. Then $K$ is not embeddable into $\R^d$. In fact, a slightly stronger conclusion holds: there is no \emph{almost-embedding}  $f\colon |K|\to \R^d$, i.e., no continuous map such that the images of disjoint simplices of $K$ are disjoint.
\end{proposition}

Another basic fact is the following result (for a short proof see, for instance,   \cite[Example~3.5] {Melikhov:vanKampenRelatives-2009}).

\begin{proposition}[\cite{vanKampen:KomplexeInEuklidischenRaeumen-1932,Flores:NichtEinbettbar-1933}]
\label{obs-nonzero-skeleton}
For every $k\geq 0$, $\obs{2k}\left(\skelsim{k}{2k+2}\right)\neq 0$.
\end{proposition}
As a consequence, one obtains Proposition~\ref{prop:vK-F}, and in fact the slightly stronger statement that $\skelsim{k}{2k+2}$ does not admit an almost-embedding into $\R^{2k}$.

\subsection{Homological Almost-Embeddings and a Van Kampen--Flores Result}
For the proof of Theorem~\ref{t:main}, we wish to replace
homotopy-theoretic notions (like $k$-connectedness) by homological
assumptions (bounded Betti numbers). The simple but useful observation
that allows us to do this is that in the standard proof of
Proposition~\ref{prop:obstruction-almost-embeddings}, which is based
on (co)homological arguments, maps can be replaced by suitable chain
maps at every step.\footnote{This observation was already used in
  \cite{Wagner:MinorsRandomExpandingHypergraphs-2011} to study the
  (non-)embeddability of certain simplicial complexes. What we call a
  \emph{homological almost-embedding} in the present paper corresponds
  to the notion of a \emph{homological minor} used in
  \cite{Wagner:MinorsRandomExpandingHypergraphs-2011}.}  The
appropriate analogue of an almost-embedding is the following.
\begin{definition}\label{d:homrep}
Let $\Rspace$ be a (nonempty) topological space, $K$ be a simplicial complex, and consider
a chain map\footnote{We recall that a chain map $\gamma\colon C_\ast
\rightarrow D_\ast$ between chain complexes is simply a sequence of
homomorphisms $\gamma_n\colon C_n\rightarrow D_n$ that commute with the
respective boundary operators, $\gamma_{n-1}\circ
\partial_{C}=\partial_{D}\circ \gamma_n$.} $\gamma\colon C_\ast(K)\rightarrow
C_\ast(\Rspace)$ from the simplicial chains in $K$ to singular chains in
$\Rspace$.
\begin{enumerate}
\item[\textup{(i)}] The chain map $\gamma$ is called \emph{\nontrivial}\footnote{If we consider augmented chain complexes with chain groups also in dimension $-1$, 
  then being \nontrivial\ is equivalent to requiring that the generator of
$C_{-1}(K)\cong \Z_2$ (this generator corresponds to the empty simplex in $K$)
is mapped to the generator of $C_{-1}(\Rspace)\cong \Z_2$.} if the image of every vertex of $K$ is a finite set of points
  in~$\Rspace$ \textup{(}a 0-chain\textup{)} of \emph{odd} cardinality. 
  \item[(ii)] The chain map $\gamma$ is called a  \emph{homological
    almost-embedding} of a simplicial complex $K$ in $\Rspace$ if it is \nontrivial\ and if, additionally, the following holds: whenever $\sigma$ and $\tau$ are disjoint simplices of $K$, their image chains
  $\gamma(\sigma)$ and $\gamma(\tau)$ have disjoint supports, where
  the support of a chain is the union of (the images of) the singular simplices 
  with nonzero coefficient in that chain.
\end{enumerate}
\end{definition}

\begin{remark} Suppose that $f\colon|K|\rightarrow \R^d$ is a continuous map.
\begin{enumerate}
\item[(i)] The induced chain map\footnote{The induced chain map is defined as follows: We assume that we have fixed a total ordering of the vertices of $K$.
For a $p$-simplex $\sigma$ of $K$, the ordering of the vertices induces a homeomorphism $h_\sigma\colon |\Delta_p|\rightarrow |\sigma|\subseteq |K|$. The
image $f_\sharp(\sigma)$ is defined as the singular $p$-simplex $f\circ h_{\sigma}$.}
$f_\sharp\colon C_\ast(K)\rightarrow C_\ast(\R^d)$ is \nontrivial.
\item[(ii)] If $f$ is an almost-embedding then the induced chain map is a
  homological almost-embedding. 
\end{enumerate}
Moreover, note that without the requirement of being \nontrivial, we could simply take the constant zero chain map, for which the second requirement is trivially satisfied.
\end{remark}

We have the following analogue of
Proposition~\ref{prop:obstruction-almost-embeddings} for homological
almost-embeddings.
\begin{proposition}
\label{prop:obstruction-no-ae}
Suppose that $K$ is a finite simplicial complex with $\obs{d}(K)\neq 0$. Then
$K$ does not admit a homological almost-embedding in $\R^d$.
\end{proposition}

As a corollary, we get the following result, which underlies our proof of Theorem~\ref{t:main}.
\begin{corollary}\label{c:nohomrep}
  For any $k \ge 0$, the $k$-skeleton $\skelsim{k}{2k+2}$ of the
  $(2k+2)$-dimensional simplex has no homological almost-embedding in
  $\R^{2k}$.
\end{corollary}

We conclude this subsection by two facts that are not
  needed for the proof of the main result but are useful for the
  presentation of our method in Section~\ref{s:history}.

  If the ambient dimension $d = 2k + 1$ is odd, we can immediately see
  that $\skelsim{k+1}{2k+4}$ has no homological almost-embedding in
  $\R^{2k+1}$ since it has no homological almost-embedding in
  $\R^{2k+2}$; this result can be slightly improved:

\begin{corollary}\label{c:nohomrep-odd}
  For any $d \ge 0$, the $\lceil d/2\rceil$-skeleton $\skelsim{\lceil
    d/2\rceil}{d+2}$ of the $(d+2)$-dimensional simplex has no
  homological almost-embedding in $\R^d$.
\end{corollary}

\begin{proof}
  The statement for even $d$ is already covered by the case $k = d/2$
  of Corollary~\ref{c:nohomrep}, so assume that $d$ is odd and write
  $d=2k + 1$. If $K$ is a finite simplicial complex with
  $\obs{d}(K)\neq 0$ and if $CK$ is the cone over $K$ then
  $\obs{d+1}(CK)\neq 0$ (for a proof, see, for instance,
  \cite[Lemma~8]{BestvinaKapovichKleiner:vanKampenDiscreteGroups-2002}).
  Since we know that $\obs {2k}(\skelsim{k}{2k+2}) \neq 0$ it follows
  that $\obs{2k+1}(C\skelsim{k}{2k+2}) \neq 0$.
%
%
  Consequently, $\obs{2k+1}(\skelsim{k+1}{2k+3}) \neq 0$ since
  $C\skelsim{k}{2k+2}$ is a subcomplex of $\skelsim{k+1}{2k+3}$ and
  there exists an equivariant map from the deleted product of the
  subcomplex to the deleted product of the complex.
  Proposition~\ref{prop:obstruction-no-ae} then implies that
  $\skelsim{k+1}{2k+3}$ admits no homological almost-embedding in
  $\R^{2k+1}$.
\end{proof}

The next fact is the following analogue of Radon's lemma, proved in the next
subsection along the proof of
Proposition~\ref{prop:obstruction-no-ae}.

\begin{lemma}[Homological Radon's lemma]
\label{l:HomRadon}
For any $d\geq 0$, $\obs{d}(\partial \simplex{d+1})\neq 0$. 
Consequently, the boundary of $(d+1)$-simplex $\partial \simplex{d+1}$ admits
no homological almost-embedding in $\R^d$.
\end{lemma}

\subsection{Deleted Products and Obstructions}
\label{sec:obstructions}
Here, we review the standard proof of
Proposition~\ref{prop:obstruction-almost-embeddings} and explain how to adapt
it to prove Proposition~\ref{prop:obstruction-no-ae}, which will
follow from Lemma~\ref{lem:homological-Gauss-map} and
Lemma~\ref{lem:obs-no-map-or-chain-map}~(b) below. The reader unfamiliar with
cohomology and willing to accept
Proposition~\ref{prop:obstruction-no-ae} can safely proceed to
Section~\ref{s:history}.

\paragraph{$\Z_2$-spaces and equivariant maps.} We begin by recalling some basic notions of equivariant topology: An \emph{action} of the group $\Z_2$ on a space $X$ is given by an automorphism $\nu\colon X\rightarrow X$ such that $\nu\circ \nu=1_X$; the action is \emph{free} if $\nu$ does not have any fixed points. If $X$ is a simplicial complex (or a cell complex), then the action is called simplicial (or cellular) if it is given by a simplicial (or cellular) map. A space with a given (free) $\Z_2$-action is also called a (free) $\Z_2$-space. 

A map $f\colon X \rightarrow Y$ between $\Z_2$-spaces $(X,\nu)$ and $(Y,\mu)$ is called \emph{equivariant} if it commutes with the respective $\Z_2$-actions, i.e., $f\circ \nu=\mu\circ f$. Two equivariant maps $f_0,f_1\colon X\rightarrow Y$ are \emph{equivariantly homotopic} if there exists a homotopy
$F\colon X\times [0,1]\to Y$ such that all intermediate maps $f_t:=F(\cdot,t)$, $0\leq t\leq 1$, are equivariant.

A $\Z_2$-action $\nu$ on a space $X$ also yields a $\Z_2$-action on the chain complex $C_\ast(X)$, given by the induced chain map $\nu_\sharp\colon C_\ast(X)\rightarrow C_\ast(X)$ (if $\nu$ is simplicial or cellular, respectively, then this remains true if we consider the simplicial or cellular chain complex of $X$ instead of the singular chain complex), and if  
$f\colon X\to Y$ is an equivariant map between $\Z_2$-spaces then the induced chain map is also equivariant (i.e., it commutes with the $\Z_2$-actions on the chain complexes).

\paragraph{Spheres.} Important examples of free $\Z_2$-spaces are the standard spheres $\s^d$, $d\geq 0$, with the action given by antipodality, $x\mapsto -x$.
There are natural inclusion maps $\s^{d-1}\hookrightarrow \s^d$, which are
equivariant. Antipodality also gives a free $\Z_2$-action on the union
$\s^\infty=\bigcup_{d\geq 0} \s^d$, the infinite-dimensional sphere. Moreover,
one can show that $\s^\infty$ is contractible, and from this it is not hard to
deduce that $\s^\infty$ is a universal $\Z_2$-space, in the following sense
(see~\cite{milnor56} or also \cite[Prop.~8.16 and
Thm.~8.17]{Kozlov:CombinatorialAlgebraicTopology-2008} for a more detailed textbook treatment). 
\begin{proposition}
\label{prop:sinfinity-maps}
If $X$ is any cell complex with a free cellular $\Z_2$-action, then there exists an equivariant map $f\colon X\to \s^\infty$. 
Moreover, any two equivariant maps $f_0,f_1\colon X\to \s^\infty$ are equivariantly homotopic.
\end{proposition}  
Any equivariant map $f\colon X\to \s^\infty$ induces a \nontrivial\ equivariant chain map $f_\sharp\colon C_\ast(X)\to C_\ast(\s^\infty)$. A simple
fact that will be crucial in what follows is that Proposition~\ref{prop:sinfinity-maps} has an analogue on the level of chain maps.

We first recall the relevant notion of homotopy between chain maps: Let $C_\ast(X)$ and $C_\ast(Y)$ be (singular or simplicial, say) chain complexes,
and let  
$\varphi,\psi\colon C_\ast(X)\to C_\ast(Y)$ be chain maps. A \emph{chain homotopy} $\eta$ between $\varphi$ and $\psi$ is a family of homomorphisms $\eta_j\colon C_j(X)\rightarrow C_{j+1}(Y)$ such that
$$\varphi_j-\psi_j =\partial_{j+1}^Y\circ \eta_j+\eta_{j-1}\circ \partial_j^X$$
for all $j$.\footnote{Here, we use subscripts and superscripts on the boundary operators to emphasize which dimension and which chain complex they belong to; often, these indices are dropped and one simply writes $\varphi-\psi=\partial \eta+\eta\partial$.} If $X$ and $Y$ are $\Z_2$-spaces then a chain homotopy is called equivariant if it commutes with the (chain maps induced by) the $\Z_2$-actions.\footnote{We also recall that if $f,g\:X\to Y$ are (equivariantly) homotopic then the induced chain maps are (equivariantly) chain homotopic. Moreover, chain homotopic maps induce \emph{identical} maps in homology and cohomology.}

\begin{lemma}
\label{lem:chain-homotopy}
If $X$ is a cell complex with a free cellular $\Z_2$-action then any two \nontrivial\ equivariant chain maps $\varphi,\psi\colon C_\ast(X)\rightarrow C_\ast(\s^\infty)$
are equivariantly chain homotopic.\footnote{We stress that we work with the cellular chain complex for $X$. }
\end{lemma}
\begin{proof}[Proof of Lemma~\ref{lem:chain-homotopy}] 
Let the $\Z_2$-action on $X$ be given by the automorphism $\nu\colon X\to X$.
For each dimension $i\geq 0$, the action partitions the $i$-dimensional
cells of $X$ (the basis elements of $C_i(X)$) into pairs $\sigma, \nu(\sigma)$.
For each such pair, we arbitrarily pick one of the cells and call it the
representative of the pair.

We define the desired equivariant chain homotopy $\eta$ between $\varphi$ and $\psi$ by induction on the dimension, using the fact that all reduced homology groups of $\s^\infty$ are zero. 
(This just mimics the argument for the existence of an equivariant
homotopy, which uses the contractibility of $\s^\infty$.)

We start the induction in dimension at $j=-1$ (and for convenience, we also use the convention that all chain groups, chain maps, and $\eta_i$ are understood to be zero in dimensions $i<-1$). Since we assume that both $\varphi$ and $\psi$ are \nontrivial, we have that $\varphi_{-1},\psi_{-1}\colon C_{-1}(X)\to C_{-1}(\s^\infty)$ are identical, and we set $\eta_{-1}\colon C_{-1}(X)\rightarrow C_0(\s^\infty)$ to be zero. 

Next, assume inductively that equivariant homomorphisms $\eta_{i}\colon C_{i}(X)\rightarrow C_i(\s^\infty)$ have already been defined for $i<j$ and satisfy
\begin{equation}
\label{eq:chain-homotopy}
\varphi_{i}-\psi_{i}=\eta_{i-1}\circ \partial + \partial \circ \eta_i
\end{equation}
for all $i<j$ (note that initially, this holds true for $j=0$).

Suppose that $\sigma$ is a $j$-dimensional cell of $X$ representing a pair $\sigma,\nu(\sigma)$. Then $\partial \sigma \in C_{j-1}(X)$, and so $\eta_{j-1}(\partial \sigma) \in C_j(\s^\infty)$ is already defined. We are looking for a suitable chain $c\in C_{j+1}(\s^\infty)$ which we can take to be $\eta_j(\sigma)$ in order to satisfy the chain homotopy relation (\ref{eq:chain-homotopy}) also for $i=j$, such a chain $c$ has to satisfy $\partial c=b$, where
$$b:=\varphi_j(\sigma)-\psi_j(\sigma)-\eta_{j-1}(\partial(\sigma)).$$
To see that we can find such a $c$, we compute
\begin{eqnarray*} 
\partial b &= & \partial\varphi_j(\sigma)-\partial \psi_j(\sigma)-\partial \eta_{j-1}(\partial(\sigma))\\
& = &\varphi_{j-1}(\partial\sigma) - \psi_{j-1}(\partial\sigma) -\Big(      \varphi_{j-1}(\partial\sigma)-\psi_{j-1}(\partial\sigma)-\eta_{j-2}(\partial\partial\sigma)       \Big) = 0
\end{eqnarray*}
Thus, $b$ is a cycle, and since $H_j(\s^\infty)=0$, $b$ is also a boundary. Pick an arbitrary chain $c\in C_{j+1}(\s^\infty)$ with $\partial c=b$ and set $\eta_j(\sigma):=c$ and $\eta_j(\nu(\sigma)):=\nu_\sharp(c)$. We do this for all representative $j$-cells $\sigma$ and then extend $\eta_j$ by linearity. By definition, $\eta_j$ is equivariant and (\ref{eq:chain-homotopy}) is now satisfied also for $i=j$.
This completes the induction step and hence the proof.\end{proof}

\paragraph{Deleted products and Gauss maps.} Let $K$ be a finite simplicial complex. Then the Cartesian product $K\times K$ is a 
cell complex whose cells are the Cartesian products of pairs of simplices of $K$. The (\emph{combinatorial}) \emph{deleted product} $\delprod{K}$
of $K$ is defined as the polyhedral subcomplex of $K\times K$ whose cells are the products of vertex-disjoint pairs of simplices of $K$, 
i.e., $\delprod{K}:=\{\sigma \times \tau: \sigma,\tau \in K, \sigma \cap \tau=\emptyset\}$. The deleted product is equipped with a natural
free $\Z_2$-action that simply exchanges coordinates, $(x,y)\mapsto (y,x)$. Note that this action is cellular since each cell $\sigma\times \tau$
is mapped to $\tau\times \sigma$.

\begin{lemma}
\label{lem:Gauss-map}
If $f \colon |K| \hookrightarrow \R^d$ is an embedding (or, more generally, an almost-embedding) then\footnote{We remark that a classical result due to Haefliger and Weber \cite{Haefliger:PlongementsDifferentiablesDomaineStable-1962,Weber:PlongementsPolyedresDomaineMetastable-1967} asserts that if $\dim K \leq (2d-3)/3$ (the so-called \emph{metastable range}) then the existence of an equivariant map from $\delprod{K}$ to $\s^{d-1}$ is also \emph{sufficient} for the existence of an embedding $K\hookrightarrow \R^d$ (outside the metastable range, this fails); see \cite{Skopenkov:EmbeddingKnottingManifoldsEuclideanSpaces-2008} for further background.} there exists
an equivariant map $\tilde f\colon \delprod{K}\to S^{d-1}$.
\end{lemma}
\begin{proof}
Define $
\tilde f(x,y):=\frac{f(x)-f(y)}{\|f(x)-f(y)\|}$. This map, called the \emph{Gauss map}, is clearly equivariant.
\end{proof}

For the proof of Proposition~\ref{prop:obstruction-no-ae}, we use the following analogue of Lemma~\ref{lem:Gauss-map}.
\begin{lemma}
\label{lem:homological-Gauss-map}
Let $K$ be a finite simplicial complex. If $\gamma\colon C_\ast(K)\rightarrow
C_\ast(\R^d)$ is a homological almost-embedding
then there is a \nontrivial\ equivariant chain map (called the \emph{Gauss chain map})
$\tilde{\gamma}\colon C_\ast(\delprod{K})\to C_\ast(\s^{d-1})$.
\end{lemma}
The proof of this lemma is not difficult but a bit technical, so we postpone it until the end of this section.

\paragraph{Obstructions.} Here, we recall a standard method for proving the non-existence of equivariant maps between 
$\Z_2$-spaces. The arguments are formulated in the language of cohomology, and, as we will see, what they actually establish is the 
non-existence of \nontrivial\ equivariant chain maps.

Let $K$ be a finite simplicial complex and let $\delprod{K}$ be its (combinatorial) deleted product. By Proposition~\ref{prop:sinfinity-maps}, there exists an equivariant map $G_K\colon \delprod{K} \rightarrow \s^\infty$, which is unique up to equivariant homotopy. By factoring out the action of $\Z_2$, this induces a map $\overline{G}_K\colon\overline{K}\rightarrow \RP^\infty$ between the quotient spaces $\overline{K}=\delprod{K}/\Z_2$ and $\RP^\infty=\s^\infty/\Z_2$ (the infinite-dimensional real projective space), and the homotopy class of the map $\overline{G}_K$ depends only\footnote{We stress that this does not mean 
that there is only one homotopy class of continuous maps $\overline{K}\rightarrow \RP^\infty$; indeed, there exist such maps that do not come from equivariant maps $\delprod{K}\rightarrow \s^\infty$, for instance the constant map that maps all of $\overline{K}$ to a single point.} on $K$. Passing to cohomology, there is a uniquely defined induced homomorphism
$$\overline{G}^\ast_K\colon H^\ast(\RP^\infty)\rightarrow H^\ast (\overline{K}).$$
It is known that $H^d(\RP^\infty)\cong \Z_2$ for every $d\geq 0$. Letting $\xi^d$ denote the unique generator of $H^d(\RP^\infty)$,
there is a uniquely defined cohomology class 
$$\obs{d}(K):=\overline{G}^\ast_K(\xi^d),$$
called the \emph{van Kampen obstruction} (with $\Z_2$-coefficients) to embedding $K$ into $\R^d$. For more details and background regarding the van Kampen obstruction, we refer the reader to \cite{Melikhov:vanKampenRelatives-2009}.

The basic fact about the van Kampen obstruction (and the reason for its name) is that $K$ does not embed (not even almost-embed) into $\R^d$ if $\obs{d}(K)\neq 0$ (Proposition~\ref{prop:obstruction-almost-embeddings}). This follows from Lemma~\ref{lem:Gauss-map} and Part~(a) of the following lemma:
\begin{lemma} \label{lem:obs-no-map-or-chain-map}
Let $K$ be a simplicial complex and suppose that $\obs{d}(K)\neq 0$.
\begin{enumerate}
\item[\textup{(a)}]
Then there is no  equivariant map $\delprod{K}\rightarrow \s^{d-1}$.
\item[\textup{(b)}]
In fact, there is no \nontrivial\ equivariant chain map $C_\ast(\delprod{K})\rightarrow C_\ast(\s^{d-1})$.
\end{enumerate}
\end{lemma}
Together with Lemma~\ref{lem:homological-Gauss-map}, Part~(b) of the lemma also
implies Proposition~\ref{prop:obstruction-no-ae}, as desired.
The simple observation underlying the proof of Lemma~\ref{lem:obs-no-map-or-chain-map} is the following

\begin{observation}
Suppose $\varphi\colon C_\ast(\delprod{K})\to C_\ast(\s^\infty)$ is a \nontrivial\ equivariant chain map (not necessarily induced by a continuous map). 
By factoring out the action of $\Z_2$, $\varphi$ induces a chain map $\overline{\varphi}\colon C_\ast(\overline{K})\rightarrow C_\ast(\RP^\infty)$.
The induced homomorphism in cohomology
$$\overline{\varphi}^\ast\colon H^\ast(\RP^\infty)\to H^\ast(\overline{K})$$
is \emph{equal} to the homomorphism $\overline{G}_K^\ast$ used in the definition of the Van Kampen obstruction, hence in particular
$$\obs{d}(K)=\overline{\varphi}^\ast(\xi^d).$$
\end{observation}
\begin{proof}
By Lemma~\ref{lem:chain-homotopy}, $\varphi$ is equivariantly chain homotopic to the \nontrivial\ equivariant chain map $(G_K)_\sharp$ induced by the map $G_K$.
Thus, after factoring out the $\Z_2$-action, the chain maps $\overline{\varphi}$ and $(\overline{G}_K)_\sharp$ from $C_\ast(\overline{K})$ to $C_\ast(\RP^\infty)$ are chain homotopic, and so induce identical homomorphisms in cohomology. 
\end{proof}

\begin{proof}[Proof of Lemma~\ref{lem:obs-no-map-or-chain-map}]
If there exists an equivariant map $f\colon \delprod{K}\to \s^{d-1}$, then the induced chain map $f_\sharp\colon C_\ast(\delprod{K})\to C_\ast(\s^{d-1})$ is equivariant and \nontrivial, so (b) implies (a), and it suffices to prove the former. 

Next, suppose for a contradiction that $\psi\colon
C_\ast(\delprod{K})\rightarrow C_\ast(\s^{d-1})$ is a \nontrivial\ equivariant
chain map. Let $i\:\colon \s^{d-1}\rightarrow \s^\infty$ denote the inclusion
map, and let $i_\sharp\colon C_\ast(\s^{d-1})\rightarrow C_\ast(\s^\infty)$ denote
the induced equivariant, \nontrivial\ chain map. Then the composition
$\varphi=(i_\sharp\circ \psi)\colon C_\ast(\delprod{K})\to C_\ast(\s^\infty)$ is
also \nontrivial\ and equivariant, and so, by the preceding observation, for
the induced homomorphism in cohomology, we get
$$\obs{d}(K)=\overline{(i_\sharp\circ
\psi)}^\ast(\xi^d)=\overline{\psi}^\ast\left(\overline{i}^\ast(\xi^d)\right).$$
However, $\overline{i}^\ast(\xi^d)\in H^d(\RP^{d-1})=0$ (for reasons of
dimension), hence $\obs{d}(K)=0$, contradicting our assumption.  
\end{proof}

\begin{remark}
The same kind of reasoning also yields the well-known \emph{Borsuk--Ulam Theorem}, which asserts that there is no equivariant map $\s^d\to \s^{d-1}$, using the fact
that the inclusion $\overline{i}\colon \RP^d\to \RP^\infty$ (induced by the equivariant inclusion $i\colon \s^d\rightarrow \s^\infty$) has the property that $\overline{i}^\ast(\xi^d)$,
the pullback of the generator $\xi^d\in H^d(\RP^\infty)$, is \emph{nonzero}.\footnote{In fact, it is known that $H^\ast(\RP^\infty)$ is isomorphic to the polynomial ring $\Z_2[\xi]$, that $H^\ast(\RP^d)\cong\Z_2[\xi]/(\xi^{d+1})$, and that $\overline{i}^\ast$ is just the quotient map.} In fact, once again one gets a homological version of the Borsuk--Ulam theorem for free: there is no \nontrivial\ equivariant chain map $C_\ast(\s^d)\rightarrow C_\ast(\s^{d-1})$.
\end{remark}

\begin{proof}[Proof of Lemma~\ref{l:HomRadon}]
  It is not hard to see that the deleted product $\delprod{\partial \simplex{d+1}} =  \delprod{\simplex{d+1}}$ of the
boundary of $(d+1)$-simplex is combinatorially isomorphic to the boundary of a certain
convex polytope and hence homeomorphic to $\s^d$(respecting the antipodality
action), see
\cite[Exercise~5.4.3]{Matousek:BorsukUlam-2003}. Thus, the assertion
$\obs{d}(\partial \simplex{d+1})\neq 0$ follows immediately from the preceding remark
(the homological proof of the Borsuk--Ulam theorem). Together with
Proposition~\ref{prop:obstruction-no-ae}, this implies that there
is no homological almost-embedding of $\partial \simplex{d+1}$ in $\R^d$.  
\end{proof}

The proof of Proposition~\ref{prop:obstruction-no-ae} is complete, except for the following:

\begin{proof}[Proof of Lemma~\ref{lem:homological-Gauss-map}]
Once again, we essentially mimic the definition of the Gauss map on the level of chains. There is one minor technical difficulty 
due to the fact that the cells of $\delprod{K}$ are products of simplices, whereas the singular homology of spaces is based on 
maps whose domains are simplices, not products of simplices (this is the same issue that arises in the proof of K\"unneth-type
formulas in homology).

Assume that $\gamma\colon C_\ast(K)\rightarrow C_\ast(\R^d)$ is a homological
almost-embedding. The desired \nontrivial\ equivariant chain map
$\delprod{\gamma}\colon C_\ast(\delprod{K})\rightarrow C_\ast(\s^{d-1})$ will be defined as the composition of three intermediate \nontrivial\ equivariant chain maps
$$
\xymatrix@C=30pt{C_\ast(\delprod{K}) \ar@/_15pt/[rrr]_{\delprod{\gamma}=p_\sharp\circ \beta\circ \alpha} \ar[r]^{\alpha}& D_\ast \ar[r]^{\beta} & C_\ast(\delprod{\R^d}) \ar[r]^{p_\sharp} & C_\ast(\s^{d-1}).}
$$
These maps and intermediate chain complexes will be defined presently. 

We define $D_\ast$ as a chain subcomplex of the tensor product
$C_\ast(\R^d)\otimes C_\ast(\R^d)$. The tensor product chain complex
has a basis consisting of all elements of the form $s\otimes t$, where
$s$ and $t$ range over the singular simplices of $\R^d$, and we take
$D_\ast$ as the subcomplex spanned by all $s\otimes t$ for which $s$
and $t$ have disjoint supports (note that $D_\ast$ is indeed a chain
subcomplex, i.e., closed under the boundary operator, since if $s$ and
$t$ have disjoint supports, then so do any pair of simplices that
appear in the boundary of $s$ and of $t$, respectively). The chain
complex $C_\ast(\delprod{K})$ has a canonical basis consisting of
cells $\sigma\times \tau$, and the chain map $\alpha$ is defined on
these basis elements by ``tensoring'' $\gamma$ with itself, i.e.,
$$\alpha(\sigma\times \tau):=\gamma(\sigma)\otimes \gamma(\tau).$$
Since $\gamma$ is nontrivial, so is $\alpha$, the disjointness properties of $\gamma$ ensure that the image of $\alpha$ does indeed lie in $D_\ast$, and $\alpha$ is clearly $\Z_2$-equivariant.

Next, consider the Cartesian product $\R^d\times \R^d$ with the natural $\Z_2$-action given by flipping coordinates.
This action is not free since it has a nonempty set of fixed points, namely the ``diagonal'' $\Diagonal=\{(x,x):x\in \R^d\}$. 
However, the action on $\R^d\times \R^d$ restricts to a free action on the
subspace $\widetilde{\R^d}:=(\R^d\times \R^d)\setminus \Diagonal$ obtained by
removing the diagonal (this subspace is sometimes called the topological
deleted product of $\R^d$). 
Moreover, there exists an equivariant map $p\colon \delprod{\R^d}\rightarrow \s^{d-1}$ defined as follows: we identify $\s^{d-1}$ with
the unit sphere in the orthogonal complement $\Diagonal^\bot=\{(w,-w)\in \R^d:w\in \R^d\}$ and take $p\colon \delprod{\R^d}\to \s^{d-1}$ 
to be the orthogonal projection onto $\Delta^\bot$ (which sends $(x,y)$ to $\frac{1}{2}(x-y,y-x)$), followed by renormalizing,
$$
p(x,y):=\frac{\frac{1}{2}(x-y,y-x)}{\|\frac{1}{2}(x-y,y-x)\|}\in \s^{d-1}\subset \Diagonal^\bot.
$$
The map $p$ is equivariant and so the induced chain map $p_\sharp$ is equivariant and nontrivial.

It remains to define $\beta\colon D_\ast\rightarrow C_\ast(\delprod{\R^d})$. For this, we use a standard chain map
$$\textup{EML}\colon C_\ast(\R^d)\otimes C_\ast(\R^d)\to C_\ast(\R^d\times \R^d),$$
sometimes called the Eilenberg--Mac Lane chain map, and then take $\beta$ to be the restriction to $D_\ast$.

Given a basis element $s\otimes t$ of $C_\ast(\R^d)\otimes C_\ast(\R^t)$, where $s\colon \Delta_p\to \R^d$ and $t\colon \Delta_q\to \R^d$ are singular simplices, we can view $s\otimes t$ as the map $s\otimes t\colon \Delta_p\times \Delta_q\to \R^d\times \R^d$ with $(x,y)\mapsto (s(x),t(y)).$ This is almost like a singular simplex in $\R^d\times \R^d$, except that the domain is not a simplex but a prism (product of simplices). The Eilenberg--Mac Lane chain map is defined by prescribing a systematic and coherent way of triangulating products of simplices $\Delta_p\times \Delta_q$ that is consistent with taking boundaries; then
$\textup{EML}(s\otimes t)\in C_{p+q}(\R^d\times \R^d)$ is defined as the singular chain whose summands are the restrictions of
the map $\sigma\otimes \tau\colon \Delta_p\times \Delta_q$ to the
$(p+q)$-simplices that appear in the triangulation of $\Delta_p\times
\Delta_q$. We refer to \cite{GonRea05} for explicit formulas for the chain map
$\textup{EML}$. What is important for us is that the chain map $\textup{EML}$
is equivariant and \nontrivial. Both properties follow more or less directly
from the construction of the triangulation of the prisms $\Delta_p\times
\Delta_q$, which can be explained as follows: Implicitly, we assume that the
vertex sets $\{0,1,\ldots,p\}$ and $\{0,1,\ldots,q\}$ are totally ordered in
the standard way. The vertex set of $\Delta_p\times \Delta_q$ is the grid
$\{0,1,\ldots,p\}\times \{0,1,\ldots,q\}$, on which we consider the
coordinatewise partial order defined by $(x,y)\leq (x',y')$ if $x\leq x'$ and
$y\leq y'$. Then the simplices of the triangulation are all totally ordered
subsets of this partial order. Thus, if
$\sigma=\{(x_0,y_0),(x_1,y_1),\ldots,(x_r,y_r)\}$ is a simplex that appears in
the triangulation of $\Delta_p\times \Delta_q$ then the simplex
$\sigma=\{(y_0,x_0),(y_1,x_1),\ldots,(y_r,x_r)\}$ obtained by flipping all
coordinates appears in the triangulation of $\Delta_q\times \Delta_p$; see
Figure~\ref{f:EML_eq}. This implies equivariance of $\textup{EML}$ (and it is \nontrivial\ since it maps a
single vertex to a single vertex).
\end{proof}
\begin{figure}
  \begin{center}
    \includegraphics{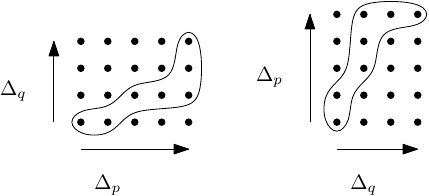}
    \caption{A simplex in a triangulation of $\Delta_p \times \Delta_q$ and its
    twin in $\Delta_q \times \Delta_p$.}
    \label{f:EML_eq}
  \end{center}
\end{figure}

\section{Helly-type theorems from non-embeddability}\label{s:history}

We now detail the technique outlined in Section~\ref{ss:outline} and illustrate
it on a few examples before formalizing its ingredients.

%
%
\paragraph{Notation.}
Given a set $X$ we let $2^X$ and $\binom{X}{k}$ denote, respectively,
the set of all subsets of $X$ (including the empty set) and the set of
all $k$-element subsets of $X$. If $f:X \to Y$ is an arbitrary map
between sets then we abuse the notation by writing $f(S)$ for $\{f(s)
\mid s \in S\}$ for any $S \subseteq X$; that is, we implicitly extend
$f$ to a map from $2^X$ to $2^Y$ whenever convenient.

\subsection{Homotopic assumptions}
\label{ss:hom_aspts}

Let $\F=\{U_1,U_2, \ldots, U_n\}$ denote a family of subsets of
$\R^d$. We assume that $\F$ has empty intersection and that any proper
subfamily of $\F$ has nonempty intersection. Our goal is to show how
various conditions on the topology of the intersections of the
subfamilies of $\F$ imply bounds on the cardinality of $\F$.  For any
(possibly empty) proper subset $I$ of $[n] = \{1,2,\ldots, n\}$ we
write $\U{I}$ for $\bigcap_{i \in [n]\setminus I} U_i$. We also put
$\U{[n]} = \R^d$.

\paragraph{Path-connected intersections in the plane.} 

Consider the case where $d=2$ and the intersections $\bigcap \G$ are
path-connected for all subfamilies $\G \subsetneq \F$. Since every
intersection of $n-1$ members of $\F$ is nonempty, we can pick, for
every $i \in [n]$, a point $p_i$ in $\U{\{i\}}$. Moreover, as every
intersection of $n-2$ members of $\F$ is connected, we can connect any
pair of points $p_i$ and $p_j$ by an arc $s_{i,j}$ inside
$\U{\{i,j\}}$. We thus obtain a drawing of the complete graph on $[n]$
in the plane in a way that the edge between $i$ and $j$ is contained
in $\U{\{i,j\}}$ (see Figure~\ref{f:k5}). If $n \ge 5$ then the
stronger form of non-planarity of $K_5$ implies that there exist two
edges $\{i,j\}$ and $\{k,\ell\}$ with no vertex in common and whose
images intersect (see
Proposition~\ref{prop:obstruction-almost-embeddings} and
Lemma~\ref{obs-nonzero-skeleton}). Since $\U{\{i,j\}} \cap
\U{\{k,\ell\}} = \bigcap \F = \emptyset$, this cannot happen and $\F$
has cardinality at most $4$.

\begin{figure}
  \begin{center}
    \includegraphics{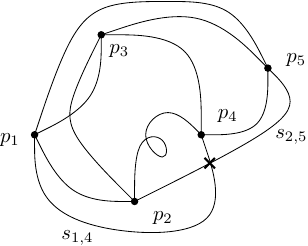}
    \caption{Two edges (arcs) with no common vertices intersect (in this case
    $s_{1,4}$ and $s_{2,5}$). The point in the intersection then belongs to all
  sets in $\F$.}
\label{f:k5}
  \end{center}
\end{figure}

\paragraph{$\lceil d/2\rceil$-connected intersections in $\R^d$.}

The previous argument generalizes to higher dimension as follows.
Assume that the intersections $\bigcap \G$ are
$\lceil d/2\rceil$-connected\footnote{Recall that a set is
  $k$-connected if it is connected and has vanishing homotopy in
  dimension $1$ to $k$.}  for all subfamilies $\G \subsetneq \F$. Then
we can build by induction a function $f$ from the
$\lceil d/2\rceil$-skeleton of $\simplex{n-1}$ to $\R^d$ in a way
that for any simplex $\sigma$, the image $f(\sigma)$ is contained
in $\U{\sigma}$. The previous case shows how to build such a function
from the $1$-skeleton of $\simplex{n-1}$. Assume that a function $f$
from the $\ell$-skeleton of $\simplex{n-1}$ is built. For every
$(\ell+1)$-simplex $\sigma$ of $\simplex{n-1}$, for every facet $\tau$ of
$\sigma$, we have $f(\tau) \subset \U{\tau} \subseteq \U{\sigma}$.
Thus, the set
\[ \bigcup_{\tau \hbox{ facet of } \sigma} f(\tau)\]
is the image of an $\ell$-dimensional sphere contained in
$\U{\sigma}$, which has vanishing homotopy of dimension $\ell$. We can extend
$f$ from this sphere to an
$(\ell+1)$-dimensional ball so that the image is still contained in
$\U{\sigma}$. This way we extend $f$ 
to the $(\ell+1)$-skeleton of $\simplex{n-1}$.


The Van Kampen-Flores theorem asserts that for any continuous function
from $\skelsim{k}{2k+2}$ to $\R^{2k}$ there exist two disjoint faces
of $\skelsim{k}{2k+2}$ whose images intersect (see
Proposition~\ref{prop:obstruction-almost-embeddings} and
Lemma~\ref{obs-nonzero-skeleton}). So, if $n \ge 2\lceil d/2
\rceil+3$, then there exist two disjoint simplices $\sigma$ and $\tau$
of $\skelsim{\lceil d/2 \rceil}{2\lceil d/2 \rceil + 2}$ such that
$f(\sigma) \cap f(\tau)$ is nonempty.  Since $f(\sigma) \cap f(\tau)$
is contained in $\U{\sigma} \cap \U{\tau} = \bigcap \F = \emptyset$,
this is a contradiction and $\F$ has cardinality at most $2\lceil
d/2\rceil +2$.

By a more careful inspection of odd dimensions, the bound $2\lceil
d/2\rceil +2$ can be improved to $d + 2$. We skip this in the
homotopic setting, but we will do so in the homological setting (which
is stronger anyway); see Corollary~\ref{c:helly_d/2-connected} below.

\paragraph{Contractible intersections.}

Of course, the previous argument works with other non-embeddability
results. For instance, if the intersections $\bigcap \G$ are
contractible for all subfamilies then the induction yields a map $f$
from the $d$-skeleton of $\simplex{n-1}$ to $\R^d$ with the property
that for any simplex $\sigma$, the image $f(\sigma)$ is contained in
$\U{\sigma}$. The topological Radon theorem~\cite{bajmoczy_barany79} (see
also~\cite[Theorem 5.1.2]{Matousek:BorsukUlam-2003}) states that for
any continuous function from $\simplex{d+1}$ to $\R^d$ there exist two
disjoint faces of $\simplex{d+1}$ whose images intersect. So, if $n
\ge d+2$ we again obtain a contradiction (the existence of two
disjoint simplices $\sigma$ and $\tau$ such that $f(\sigma) \cap
f(\tau) \neq \emptyset$ whereas $\U{\sigma} \cap \U{\tau} = \bigcap \F
= \emptyset$), and the cardinality of $\F$ must be at most $d+1$.

\subsection{From homotopy to homology}

The previous reasoning can be transposed to homology as follows.
Assume that for $i=0, 1, \ldots, k-1$ and all subfamilies $\G \subsetneq
\F$ we have $\tilde{\beta}_i(\bigcap\G)=0$. We construct a
nontrivial\footnote{See Definition~\ref{d:homrep}.} chain map $f$ from
the simplicial chains of $\skelsim{k}{n-1}$ to the singular chains of
$\R^d$ by increasing dimension:
\begin{itemize}
\item For every $\{i\}\subset [n]$ we let $p_i \in \U{\{i\}}$. This is
  possible since every intersection of $n-1$ members of $\F$ is
  nonempty. We then put $f(\{i\}) = p_i$ and extend it by linearity
  into a chain map from $\skelsim{0}{n-1}$ to $\R^d$. Notice that $f$
  is nontrivial and that for any $0$-simplex $\sigma \subseteq [n]$,
  the support of $f(\sigma)$ is contained in $\U{\sigma}$.

\item Now, assume, as an induction hypothesis, that there exists a
  nontrivial chain map $f$ from the simplicial chains of
  $\skelsim{\ell}{n-1}$ to the singular chains of $\R^d$ with the
  property that for any $(\le \ell)$-simplex $\sigma \subseteq [n], \ell < k$,
  the support of $f(\sigma)$ is contained in $\U{\sigma}$. Let
  $\sigma$ be a $(\ell+1)$-simplex in $\skelsim{\ell+1}{n-1}$.  For
  every $\ell$-dimensional face $\tau$ of $\sigma$, the support of
  $f(\tau)$ is contained in $\U{\tau} \subseteq \U{\sigma}$. It
  follows that the support of $f(\partial \sigma)$ is contained in
  $\U{\sigma}$, which has trivial homology in dimension $\ell+1$. As a
  consequence, $f(\partial \sigma)$ is a boundary in $\U{\sigma}$. We
  can therefore extend $f$ to every simplex of dimension $\ell+1$ and
  then, by linearity, to a chain map from the simplicial chains of
  $\skelsim{\ell+1}{n-1}$ to the singular chains of $\R^d$. This chain
  map remains nontrivial and, by construction, for any $(\le
  \ell+1)$-simplex $\sigma \subseteq [n]$, the support of $f(\sigma)$
  is contained in $\U{\sigma}$.
\end{itemize}
If $\sigma$ and $\tau$ are disjoint simplices of $\skelsim{k}{n-1}$
then the intersection of the supports of $f(\sigma)$ and $f(\tau)$ is
contained in $ \U{\sigma} \cap \U{\tau} = \bigcap \F = \emptyset$ and
these supports are disjoint. It follows that $f$ is not only a nontrivial chain
map, but also a homological almost-embedding in $\R^d$.  We can then use obstructions
to the existence of homological almost-embeddings to bound the
cardinality of $\F$. Specifically, since we assumed that $\F$ has empty intersection and any proper subfamily of $\F$ has nonempty intersection,
Corollary~\ref{c:nohomrep-odd} implies:

\begin{corollary}
\label{c:helly_d/2-connected}
  Let $\F$ be a family of subsets of $\R^d$ such that $\tilde
  \beta_i(\bigcap \G) = 0$ for every $\G \subsetneq \F$ and $i=0,1,
  \ldots, \lceil d/2\rceil - 1$. Then the Helly number of $\F$ is at
  most $d +2$.
\end{corollary}

The homological Radon's lemma (Lemma~\ref{l:HomRadon}) yields (noting $\partial
\Delta_{d+1} = \Delta_{d+1}^{(d)}$):

\begin{corollary}
  \label{c:helly_hom_radon}
  Let $\F$ be a family of subsets of $\R^d$ such that $\tilde
  \beta_i(\bigcap \G) = 0$ for every $\G \subsetneq \F$ and $i=0,1,
  \ldots, d - 1$. Then the Helly number of $\F$ is at most $d+1$.
\end{corollary}

\begin{remark}
  \label{r:tight_bounds}
    The following modification of Example~\ref{ex:necessary} shows
    that the two previous statements are sharp in various ways.
    First assume that for some values $k,n$ there exists some
    embedding $f$ of $\skelsim{k}{n - 1}$ into $\R^d$. Let $K_i$ be
    the simplicial complex obtained by deleting the $i$th vertex of
    $\skelsim{k}{n - 1}$ (as well as all simplices using that vertex)
    and put $U_i := f(K_i)$. The family $\F = \{U_1, \dots, U_n\}$ has
    Helly number exactly $n$, since it has empty intersection and all
    its proper subfamilies have nonempty intersection. Moreover, for
    every $\G \subseteq \F$, $\bigcap \G$ is the image through $f$ of
    the $k$-skeleton of a simplex on $|\F \setminus \G|$ vertices, and
    therefore $\tilde \beta_i(\bigcap \G) = 0$ for every $\G \subseteq
    \F$ and $i = 0,\dots, k - 1$. Now, such an embedding exists for:
\begin{description}
\item[$k=d$ and $n=d+1$,] as the
  $d$-dimensional simplex easily embeds into $\R^d$. Consequently, the
  bound of $d+1$ is best possible under the assumptions of
  Corollary~\ref{c:helly_hom_radon}.
\item[$k=d-1$ and $n=d+2$,] as we can first embed the $(d-1)$-skeleton
  of the $d$-simplex linearly, then add an extra vertex at the
  barycentre of the vertices of that simplex and embed the remaining
  faces linearly. This implies that if we relax the condition of
  Corollary~\ref{c:helly_hom_radon} by only controlling the first $d-2$
  Betti numbers then the bound of $d+1$ becomes false.  It also
  implies that the bound of $d+2$ is best possible under (a
  strengthening of) the assumptions of
  Corollary~\ref{c:helly_d/2-connected}.
\end{description}
(Recall that, as explained in Example~\ref{ex:necessary}, the $\lceil
d/2\rceil - 1$ in the assumptions of
Corollary~\ref{c:helly_d/2-connected} cannot be reduced without
allowing unbounded Helly numbers.)
\end{remark}

\paragraph{Constrained chain map.}

Let us formalize the technique illustrated by the previous example. We
focus on the homological setting, as this is what we use to prove
Theorem~\ref{t:main}, but this can be easily transposed to homotopy.

  Considering a slightly more general situation, we let
  $\F=\{U_1,U_2, \ldots, U_n\}$ denote a family of subsets of some
  topological space $\Rspace$.  As before for any (possibly empty)
  proper subset $I$ of $[n] = \{1,2,\ldots, n\}$ we write $\U{I}$ for
  $\bigcap_{i \in [n]\setminus I} U_i$ and we put $\U{[n]} = \Rspace$.

Let $K$ be a simplicial complex and let $\gamma: C_*(K) \to C_*(\Rspace)$
be a chain map from the simplicial chains of $K$ to the singular
chains of $\Rspace$. We say that $\gamma$ is \emph{constrained by
  $(\F,\Phi)$} if:
\begin{itemize}
\item[(i)] $\Phi$ is a map from $K$ to $2^{[n]}$ such that $\Phi(\sigma
  \cap \tau) = \Phi(\sigma) \cap \Phi(\tau)$ for all $\sigma, \tau \in K$ and
  $\Phi(\emptyset)=\emptyset$.

\item[(ii)] For any simplex $\sigma \in K$, the support of
  $\gamma(\sigma)$ is contained in $\U{\Phi(\sigma)}$.
\end{itemize}
\begin{figure}
  \begin{center}
    \includegraphics{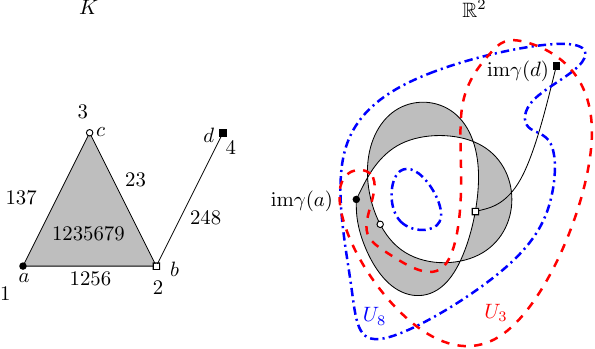}
    \caption{An example of a constrained map $\gamma\colon K \to \R^2$. A
      label at a face $\sigma$ of $K$ denotes $\Phi(\sigma)$. Note, for example,
    that the support of $\gamma(\{a,b,c\})$ needn't be a triangle since we work
  with chain maps. Constrains by $\Phi$ mean that a set $U_i$ must contain
cover images of all faces without label $i$. It is demonstrated by $U_3$ and
$U_8$ for example.}
\label{f:con}
  \end{center}
\end{figure}
See Figure~\ref{f:con}. We also say that a chain map $\gamma$ from $K$
is \emph{constrained by $\F$} if there exists a map $\Phi$ such that
$\gamma$ is constrained by $(\F,\Phi)$. In the above constructions, we
simply set $\Phi$ to be the identity. As we already saw, constrained
chain maps relate Helly numbers to homological almost-embeddings (see
Definition~\ref{d:homrep}) via the following observation:

\begin{lemma}\label{l:homrep}
  Let $\gamma: C_*(K) \to C_*(\Rspace)$ be a nontrivial chain map
  constrained by $\F$. If $\bigcap\F =\emptyset$ then $\gamma$ is a
  homological almost-embedding of $K$.
\end{lemma}
\begin{proof}
  Let $\Phi:K \to 2^{[n]}$ be such that $\gamma$ is constrained by
  $(\F,\Phi)$. Since $\gamma$ is nontrivial, it remains to check that
  disjoint simplices are mapped to chains with disjoint support. Let
  $\sigma$ and $\tau$ be two disjoint simplices of $K$. The supports
  of $\gamma(\sigma)$ and $\gamma(\tau)$ are contained, respectively,
  in $\U{\Phi(\sigma)}$ and $\U{\Phi(\tau)}$, and
  \[\U{\Phi(\sigma)} \cap \U{\Phi(\tau)} = \U{\Phi(\sigma) \cap
    \Phi(\tau)} = \U{\Phi(\sigma \cap \tau)}= \U{\Phi(\emptyset)} =
  \U{\emptyset} = \bigcap\F.\]
  Therefore, if $\bigcap\F =\emptyset$ then $\gamma$ is a homological
  almost-embedding of $K$.
\end{proof}

\subsection{Relaxing the connectivity assumption}\label{s:mat}
\label{ss:matousek}

In all the examples listed so far, the intersections $\bigcap\G$ must
be connected. Matou\v{s}ek~\cite{m-httucs-97} relaxed this condition
into ``having a bounded number of connected components'', the
assumptions then being on the topology of the components, by using
Ramsey's theorem. The gist of our proof is to extend his idea to allow
a bounded number of homology classes not only in the first dimension
but in \emph{any} dimension. Let us illustrate how Matou\v{s}ek's idea
works in two dimension:
\begin{theorem}[{\cite[Theorem 2 with $d = 2$]{m-httucs-97}}]
  For every positive integer $b$ there is an integer $h(b)$ with the
  following property.  If $\F$ is a finite family of subsets of $\R^2$
  such that the intersection of any subfamily has at most $b$
  path-connected components, then the Helly number of $\F$ is at most
  $h(b)$.
\end{theorem}

Let us fix $b$ from above and assume that for any
subfamily $\G \subsetneq \F$ the intersection $\bigcap \G$ consists of
at most $b$ path-connected components and that $\bigcap \F = \emptyset$.
We start, as before, by picking
for every $i \in [n]$, a point $p_i$ in $\U{\{i\}}$. This is possible
as every intersection of $n-1$ members of $\F$ is nonempty. Now, if
we consider some pair of indices $i,j \in [n]$, the points $p_i$ and
$p_j$ are still in $\U{\{i,j\}}$ but may lie in different connected
components. It may thus not be possible to connect $p_i$ to $p_j$
\emph{inside} $\U{\{i,j\}}$. If we, however, consider $b+1$ indices
$i_1, i_2, \ldots, i_{b+1}$ then all the points $p_{i_1}, p_{i_2},
\ldots, p_{i_{b+1}}$ are in $\U{\{i_1, i_2, \ldots, i_{b+1}\}}$ which has
at most $b$ connected components, so at least one pair among of these
points can be connected by a path inside $\U{\{i_1, i_2, \ldots,
  i_{b+1}\}}$.  Thus, while we may not get a drawing of the complete
graph on $n$ vertices we can still draw many edges.

To find many vertices among which every pair can be connected we will
use the hypergraph version of the classical theorem of Ramsey:


\begin{theorem}[Ramsey~\cite{ramsey30}]
  \label{t:Ramsey} 
  For any $x$, $y$ and $z$ there is an integer $R_x(y,z)$ such that
  any $x$-uniform hypergraph on at least $R_x(y,z)$ vertices colored
  with at most $y$ colors contains a subset of $z$ vertices inducing a
  monochromatic sub-hypergraph.
\end{theorem}

\noindent
From the discussion above, for any $b+1$ indices $i_1 < i_2 < \ldots < i_{b+1}$ there exists a
pair $\{k,\ell\} \in \binom{[b+1]}{2}$ such that $p_{i_k}$ and
$p_{i_\ell}$ can be connected inside $\U{\{i_1, i_2, \ldots,
  i_{b+1}\}}$. Let us consider the $(b+1)$-uniform hypergraph on $[n]$
and color every set of indices $i_1 < i_2 < \ldots < i_{b+1}$ by one of
the pairs in $\binom{[b+1]}{2}$ that can be connected inside
$\U{\{i_1, i_2, \ldots, i_{b+1}\}}$ (if more than one pair can be
connected, we pick one arbitrarily). Let $t$ be some integer to be
fixed later. By Ramsey's theorem, if $n \ge
R_{b+1}\pth{\binom{b+1}2,t}$ then there exist a pair $\{k,\ell\} \in
\binom{[b+1]}{2}$ and a subset $T \subseteq [n]$ of size $t$ with the
following property: for any $(b+1)$-element subset $S \subset T$, the
points whose indices are the $k$th and $\ell$th indices of $S$ can be
connected inside $\U{S}$.

Now, let us set $t=5+\binom52(b-1) = 10b-5$. We claim that we can find
five indices in $T$, denoted $i_1, i_2, \ldots, i_5$, and, for each
pair $\{i_u,i_v\}$ among these five indices, some $(b+1)$-element
subset $Q_{u,v} \subset T$ with the following properties:
\begin{itemize}
\item [(i)] $i_u$ and $i_v$ are precisely in the $k$th and $\ell$th
  position in $Q_{u,v}$, and
\item [(ii)] for any $1 \le u,v,u',v' \le 5$, \quad $Q_{u,v} \cap
  Q_{u',v'} = \{i_u,i_v\} \cap \{i_{u'},i_{v'}\}$.
\end{itemize}
We first conclude the argument, assuming that we can obtain such indices and
sets. Observe that from the construction of $T$, the
$i_u$'s and the $Q_{u,v}$'s we have the following property: for any
$u,v \in [5]$, we can connect $p_{i_u}$ and $p_{i_v}$ inside
$\U{Q_{u,v}}$. This gives a drawing of $K_5$ in the plane. Since $K_5$
is not planar, there exist two edges with no vertex in common, say
$\{u,v\}$ and $\{u',v'\}$, that cross. This intersection point must
lie in
\[ \U{Q_{u,v}} \cap \U{Q_{u',v'}} = \U{Q_{u,v} \cap Q_{u',v'}} = \U{\{i_u,i_v\} \cap \{i_{u'},i_{v'}\}} = \U{\emptyset} = \bigcap \F = \emptyset,\]
a contradiction. 
 Hence the assumption that $n \ge
R_{b+1}\pth{\binom{b+1}2,t}$ is false and $\F$ has cardinality at most
$R_{b+1}\pth{\binom{b+1}2,10b-5}-1$, which is our $h(b)$.

\paragraph{The selection trick.}

It remains to derive the existence of the $i_u$'s and the $Q_{u,v}$'s.
It is perhaps better to demonstrate the method by a simple example to
develop some intuition before we formalize it.

\smallskip
\noindent
\emph{Example.}  Let us fix $b = 4$ and $\{ k,\ell \} = \{2,3\} \in
\binom{[4 + 1]}2$. We first make a `blueprint' for the construction
inside the rational numbers. For any two indices $u, v \in [5]$ we
form a totally ordered set $Q'_{u,v} \subseteq \Q$ of size $b + 1 = 5$
by adding three rational numbers (different from $1,\dots,5$) to the
set $\{u, v\}$ in such a way that $u$ appears at the second and $v$ at
the third position of $Q'_{u,v}$. For example, we can set $Q'_{1,4}$
to be $\{0.5; 1; 4; 4.7; 5.13\}$. Apart from this we require that we
add a different set of rational numbers for each $\{u,v\}$. Thus
$Q'_{u,v} \cap Q'_{u',v'} = \{u, v\} \cap \{u', v'\}$. Our blueprint
now appears inside the set $T' := \bigcup_{1 \leq u < v \leq 5}
Q'_{u,v}$; note that both this set $T'$ and the set $T$ in which we
search for the sets $Q_{u,v}$ have $35$ elements.  To obtain the
required indices $i_{u}$ and sets $Q_{u,v}$ it remains to consider the
unique strictly increasing bijection $\pi_0\colon T' \to T$ and set
$i_u := \pi_0(u)$ and $Q_{u,v} := \pi_0(Q'_{u,v})$.



\smallskip
\noindent
\emph{The general case.} Let us now formalize the generalization of
this trick that we will use to prove Theorem~\ref{t:main}.
Let $Q$ be a subset of $[w]$.  If $e_1 < e_2 < \ldots < e_w$ are the
elements of a totally ordered set $W$ then we call $\{e_i : i \in Q\}$
the \emph{subset selected by $Q$ in $W$}.

\begin{lemma}\label{l:rescale}
  Let $1 \le q \le w$ be integers and let $Q$ be a subset of $[w]$ of
  size $q$. Let $Y$ and $Z$ be two finite totally ordered sets and let
  $A_1,A_2, \ldots, A_r$ be $q$-element subsets of $Y$.  If $|Z| \ge
  |Y| + r(w-q)$, then there exist an injection $\pi: Y \to Z$ and $r$
  subsets $W_1, W_2, \ldots, W_r \in \binom{Z}{w}$ such that for every
  $i \in [r]$, $Q$ selects $\pi(A_i)$ in $W_i$. We can further require
  that $W_i \cap W_j = \pi(A_i \cap A_j)$ for any two $i, j \in [r]$, $i \neq
  j$.
\end{lemma}
\begin{proof}
  Let $\pi_0$ denote the monotone bijection between $Y$ and $[|Y|]$.
  For $i \in [r]$ we let $D_i$ denote a set of $w-q$ rationals,
  disjoint from $[|Y|]$, such that $Q$ selects $\pi_0(A_i)$ in $D_i
  \cup \pi_0(A_i)$. We further require that the $D_i$ are pairwise
  disjoint, and put $Z' = [|Y|] \cup \pth{\bigcup_{i \in [r]}D_i}$.
  Since $|Z| \ge |Y|+r(w-q)= |Z'|$ there exists a strictly increasing
  map $\nu: Z' \to Z$. We set $\pi := \nu \circ \pi_0$ and
  $W_i := \nu(D_i \cup \pi_0(A_i)) \in \binom{Z}{w}$. The
  desired condition is satisfied by this choice. See Figure~\ref{f:rational}.
\end{proof}

\begin{figure}
  \begin{center}
  \includegraphics{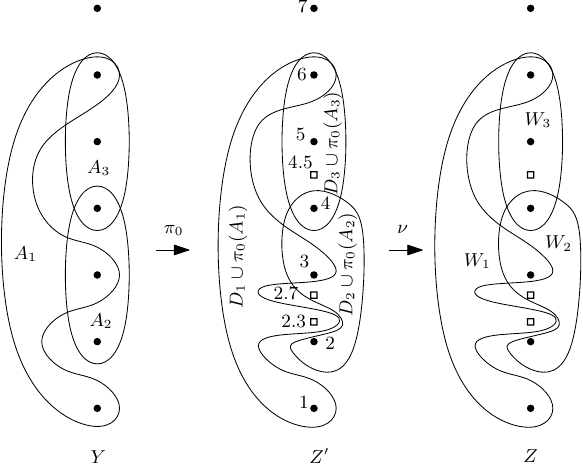}
  \caption{Illustration for the proof of Lemma~\ref{l:rescale}. We assume that
  $w = 4$ and $Q = \{1,3,4\}$.}
  \label{f:rational}
  \end{center}
\end{figure}



\section{Constrained chain maps and Helly number}\label{s:ccm}


We now generalize the technique presented in Section~\ref{s:history}
to obtain Helly-type theorems from non-embeddability results.  We will
construct constrained chain maps for arbitrary complexes.  As above,
$\F=\{U_1,U_2, \ldots, U_n\}$ denotes a family of subsets of some
topological space $\Rspace$ and for $I \subseteq [n]$ we keep the
notation $\U{I}$ as used in the previous section (see the beginning of
Subsection~\ref{ss:hom_aspts}). Note that
  although so far we only used the \emph{reduced} Betti numbers
  $\tilde{\beta}$, in this section it will be convenient to work with
  \emph{standard} (non-reduced) Betti numbers $\beta$, starting with the following proposition.

\begin{proposition}\label{p:ccm}
  For any finite simplicial complex $K$ and non-negative integer $b$
  there exists a constant $h_K(b)$ such that the following holds. For
  any finite family $\F$ of at least $h_K(b)$ subsets of a topological
  space $\Rspace$ such that $\bigcap\G \neq \emptyset$ and
  $\beta_i\pth{\cap\G} \le b$ for any $\G \subsetneq \F$ and any $0\le
  i<\dim K$, there exists a nontrivial chain map $\gamma: C_*(K) \to
  C_*(\Rspace)$ that is constrained by $\F$.
\end{proposition}

\noindent
The case $K=\skelsim{k}{2k+2}$, with $k=\lceil d/2 \rceil$ and
$\Rspace=\R^d$, of Proposition~\ref{p:ccm} implies
Theorem~\ref{t:main}.


\begin{proof}[Proof of Theorem~\ref{t:main}]
  Let $b$ and $d$ be fixed integers, let $k=\lceil d/2 \rceil$ and let
  $K=\skelsim{k}{2k+2}$. Let $h_K(b + 1)$ denote the constant from
  Proposition~\ref{p:ccm} (we plug in $b + 1$ because we need to
  switch between reduced and non-reduced Betti numbers).  Let $\F$ be
  a finite family of subsets of $\R^d$ such that
  $\tilde\beta_i\pth{\bigcap \G} \le b$ for any $\G \subsetneq \F$ and
  every $0 \le i \le \dim K = \lceil d/2 \rceil-1$, in particular
  $\beta_i\pth{\bigcap \G} \le b + 1$ for such $\G$. Let $\F^*$ denote
  an inclusion-minimal sub-family of $\F$ with empty intersection:
  $\bigcap{\F^*} = \emptyset$ and $\bigcap(\F^*\setminus\{U\}) \neq
  \emptyset$ for any $U \in \F^*$.  If $\F^*$ has size at least
  $h_K(b+1)$, it satisfies the assumptions of Proposition~\ref{p:ccm}
  and there exists a nontrivial chain map from $K$ that is
  constrained by $\F^*$. Since $\F^*$ has empty intersection, this
  chain map is a homological almost-embedding by Lemma~\ref{l:homrep}.
  However, no such homological almost-embedding exists by
  Corollary~\ref{c:nohomrep}, 
  so $\F^*$ must have size at most
  $h_K(b+1)-1$. As a consequence, the Helly number of $\F$ is bounded
  and the statement of Theorem~\ref{t:main} holds with
  $h(b,d)=h_K(b+1)-1$.
\end{proof}

\noindent
The rest of this section is devoted to proving Proposition~\ref{p:ccm}.
We proceed by induction on the dimension of $K$, Section~\ref{s:k=0}
settling the case of $0$-dimensional complexes and
Section~\ref{s:induction} showing that if Proposition~\ref{p:ccm}
holds for all simplicial complexes of dimension $i$ then it also holds
for all simplicial complexes of dimension $i+1$. As the proof of the
induction step is quite technical, as a warm-up, we provide the reader
with a simplified argument for the induction step from $i=0$ to $i=1$
in Section~\ref{s:k=1}. We let $V(K)$ and $v(K)$ denote, respectively,
the set of vertices and the number of vertices of $K$.

\subsection{Initialization ($\dim K=0$)}\label{s:k=0}

If $K$ is a $0$-dimensional simplicial complex then
Proposition~\ref{p:ccm} holds with $h_K(b) = v(K)$. Indeed, consider a
family $\F$ of at least $v(K)$ subsets of $\Rspace$ such that all proper
subfamilies have nonempty intersection. We enumerate the vertices of
$K$ as $\{v_1, v_2, \ldots, v_{v(K)}\}$ and define
$\Phi(\{v_i\})=\{i\}$; in plain English, $\Phi$ is a bijection between
the set of vertices of $K$ and $\{1,2,\ldots, v(K)\}$. We first define
$\gamma$ on $K$ by mapping every vertex $v \in K$ to a point $p(v) \in
\U{\Phi(v)}$, then extend it linearly into a chain map $\gamma:C_0(K)
\to C_0(\Rspace)$. It is clear that $\gamma$ is nontrivial and constrained
by $(\F,\Phi)$, so Proposition~\ref{p:ccm} holds when $\dim K=0$.

\subsection{Principle of the induction mechanism ($\dim K=1$)}\label{s:k=1}


As a warm-up, we now prove Proposition~\ref{p:ccm} for $1$-dimensional
simplicial complexes. While this merely amounts to reformulating
Matou\v{s}ek's proof for embeddings~\cite{m-httucs-97} in the language
of chain maps, it still introduces several key ingredients of the
induction while avoiding some of its complications. To avoid further
technicalities, we use the non-reduced version of Betti numbers here.

\medskip

Let $K$ be a $1$-dimensional simplicial complex with vertices $\{v_1,
v_2, \ldots, v_{v(K)}\}$ and assume that $\F$ is a finite family of
subsets of a topological space $\Rspace$ such that for any $\G \subsetneq
\F$, $\bigcap\G \neq \emptyset$ and $\beta_0\pth{\cap\G} \le b$. Let
$s \in \N$ denote some parameter, to be fixed later. We assume that
the cardinality of $\F$ is large enough (as a function of $s$) so
that, as argued in Subsection~\ref{s:k=0}, there exist a bijection
$\Psi:\skelsim{0}{s} \to [s+1]$ and a nontrivial chain map
$\gamma':C_*(\skelsim{0}{s}) \to C_*(\Rspace)$ constrained by $(\F,\Psi)$.
We extend $\Psi$ to $\simplex{s}$ by putting
$\Psi(\sigma)=\cup_{v\in \sigma}\Psi(v)$ for any $\sigma \in
\simplex{s}$ and $\Psi(\emptyset) = \emptyset$. Remark that for any
$\sigma, \tau \in \simplex{s}$ we have $\Psi(\sigma \cap \tau) =
\Psi(\sigma) \cap \Psi(\tau)$.
%
%

\medskip

  We now look for an injection $f$ of $V(K)$ into
  $V(\simplex{s})$ such that the chain map $\gamma' \circ
  f_\sharp\colon C_*(K^{(0)}) \to C_*(\Rspace)$ can be extended into a
  chain map $\gamma: C_*(K) \to C_*(\Rspace)$ constrained by $\F$. Let
  $e=\{u,v\}$ be an edge in $K$. If we could arrange that
  $\gamma'(f(u)+f(v))$ is a boundary in $\U{\Psi(\{f(u),f(v)\})}$ then
  we could simply define $\gamma(e)$ to be a chain in
  $\U{\Psi(\{f(u),f(v)\})}$ bounded by $\gamma'(f(u)+f(v))$ (see
  Figure~\ref{f:injectK}). Unfortunately this is too much to ask for but
  we can still follow the Ramsey-based approach of
  Subsection~\ref{s:mat}: we add ``dummy'' vertices to
  $\{\Psi(\{f(u),f(v)\})\}$ to obtain a set $W_e$ such that
  $\gamma'(f(u)+f(v))$ is a boundary in $\U{W_e}$.  If we use
  different dummy vertices for distinct edges then setting $\gamma(e)$
  to be a chain in $\U{W_e}$ bounded by $\gamma'(f(u)+f(v))$ still
  yields a chain map constrained by $\F$. We spell out the details in
  four steps.

\begin{figure}
  \begin{center}
    \includegraphics{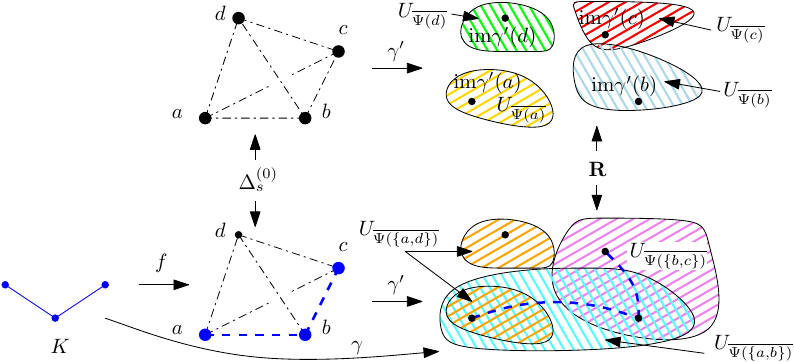}
    \caption{Injecting $V(K)$ into $V(\simplex{s})$ by $f$ in a
        way that the constrained chain map $\gamma'$ from
        $V(\simplex{s})$ (top) can give rise to a constrained chain
        map from $V(K)$ (bottom); for the sake of illustration we use
        maps instead of chain maps. The situation considered here is
        simple, for instance $\gamma'(a + b)$ is a boundary in
        $\U{\Psi(\{a,b\})}$ so $\gamma' \circ f_\sharp$ can be
        extended to the edge $\{f^{-1}(a), f^{-1}(b)\}$ of $K$. Note
        that if we wanted to use the edge $ad$, since $\gamma'(a+d)$
        is not a boundary in $\U{\Psi(\{a,d\})}$ we would need to add
        ``dummy'' elements to $\Psi(\{a,d\})$.}
    \label{f:injectK}
\end{center}
\end{figure}


%

\medskip

\begin{description}
\item[Step 1.] Any set $S$ of $2^b+1$ vertices of $\simplex{s}$
  contains two vertices $u_S,v_S \in S$ such that $\gamma'(u_S+v_S)$
  is a boundary in $\U{\Psi(S)}$.\footnote{We could require that
    $\gamma'$ sends every vertex to a point in $\U{\Psi(S)}$, i.e. is
    a chain map induced by a map, and simply argue that since
    $\U{\Psi(S)}$ has at most $b$ connected components, any $b+1$
    vertices of $\simplex{s}$ contains some pair that can be connected
    inside $\U{\Psi(S)}$. This argument does not, however, work in
    higher dimension. Since Section~\ref{s:k=1} is meant as an
    illustration of the general case, we choose to follow the general
    argument.} Indeed, notice first that for any $u \in S$, the
  support of $\gamma'(u)$ is contained in $\U{\Psi(S)}$. The
  assumption on $\F$ about bounded Betti numbers of intersections of
  subfamilies of $\F$ then ensures that there are at most $2^b$
  distinct elements
  in $H_0(\U{\Psi(S)})$, as $H_0(\U{\Psi(S)}) \simeq \Z_2^m$ for some $m
  \leq b$. Thus, there are two
  vertices $u_S,v_S \in S$ such that $\gamma'(u_S)$ and $\gamma'(v_S)$
  are in the same homology class in $H_0(\U{\Psi(S)})$. Since we
  consider homology with coefficients over $\Z_2$, the sum of two
  chains that are in the same homology class is always a boundary. In
  particular, $\gamma'(u_S+v_S) = \gamma'(u_S)+\gamma'(v_S)$ is a
  boundary in $\U{\Psi(S)}$.

\item[Step 2.] We use Ramsey's theorem (Theorem~\ref{t:Ramsey})
  to ensure a uniform ``$2$-in-$(2^b+1)$'' selection. Let $t$ be some
  parameter to be fixed in Step~3 and let $H$ denote the
  $(2^b+1)$-uniform hypergraph with vertex set $V(\simplex{s})$.
  For every hyperedge $S \in H$ there exists (by Step~1) a pair
  $Q_S\in \binom{[2^b+1]}{2}$ that selects a pair whose sum is mapped
  by $\gamma'$ to a boundary in $\U{\Psi(S)}$. We color $H$ by
  assigning to every hyperedge $S$ the ``color'' $Q_S$. Ramsey's
  theorem thus ensures that if $s \ge
  R_{2^b+1}\pth{\binom{2^b+1}{2},t}$ then there exist a set $T$ of
  $t$ vertices of $\simplex{s}$ and a pair $Q^* \in
  \binom{[2^b+1]}{2}$ so that $Q^*$ selects in \emph{any} $S \in
  \binom{T}{2^b+1}$ a pair $\{u_S,v_S\}$ such that $\gamma'(u_S+v_S)$
  is a boundary in $\U{\Psi(S)}$.

\item[Step 3.]  Now, let $r$ be the number of edges of $K$ and let
  $\sigma_1, \sigma_2, \ldots, \sigma_r$ denote the edges of $K$. We
  define
  \[ h_K(b) = R_{2^b+1}\pth{\binom{2^b+1}{2}, r(2^b-1)+v(K)}+1\]
  and assume that $s \ge h_K(b)-1$. We set the parameter $t$
  introduced in Step~2 to $t=r(2^b-1)+v(K)$. We can now apply
  Lemma~\ref{l:rescale} with $Y = V(K)$, $Z = T$, $q=2$, $w=2^b+1$,
  and $A_i = \sigma_i$ for $i \in [r]$. As a consequence, there exist
  an injection $f: V(K) \to T$ and $W_1, W_2, \ldots, W_r$ in
  $\binom{T}{2^b+1}$ such that (i) for each $i$, $Q^*$ selects
  $f(\sigma_i)$ in $W_i$, and (ii) $W_i \cap W_j = f(\sigma_i \cap
  \sigma_j)$ for $i,j \in [r], i  \neq j$.

\item[Step 4.] We define $\Phi$ by
  \[  \begin{array}{ccll}
    \Phi(\emptyset) & = & \emptyset& \\
    \Phi(\{v_i\}) & = & \Psi(f(v_i)) & \hbox{for } i=1, 2, \ldots, v(K)\\
    \Phi(\sigma_i) & = &  \Psi(W_i) & \hbox{for } i=1, 2, \ldots, r
  \end{array}\]
  We define $\gamma$ on the vertices of $K$ by putting
  $\gamma(v)=\gamma'(f(v))$ for any $v \in V(K)$. Now remark that for
  any edge $\sigma_i=\{u,v\}$ of $K$, $\gamma'(f(u)+f(v))$ is a
  boundary in $\U{\Psi(W_i)}$; this follows from the definition of $T$
  and the fact that $Q^*$ selects $\{f(u),f(v)\}$ in $W_i$. We can
  therefore define $\gamma(\{u,v\})$ to be some (arbitrary) chain in
  $\U{\Psi(W_i)}$ with boundary $\gamma'(f(u)+f(v))$. We then extend
  this map linearly into a chain map $\gamma:C_*(K) \to C_*(\Rspace)$.
\end{description}

To conclude the proof of Proposition~\ref{p:ccm} for $1$-dimensional
complexes it remains to check that the chain map $\gamma$ and the
function $\Phi$ defined in Step~4 have the desired properties.

\begin{observation}
  $\gamma$ is a nontrivial chain map constrained by $(\F,\Phi)$.
\end{observation}
\begin{proof}
  First, it is clear from the definition that $\gamma$ is a chain map.
  Moreover, the definition of $\gamma'$ ensures that for every vertex
  $v \in K$ the support of $\gamma(v)$ is a finite set of points with
  odd cardinality. So $\gamma$ is indeed a nontrivial chain map.

  The map $\Phi$ is from $K$ to $2^{[s+1]}$ and $\Phi(\emptyset)$ is
  by definition the empty set. The next property to check is that the
  identity $\Phi(\sigma \cap \tau) = \Phi(\sigma) \cap \Phi(\tau)$
  holds for all $\sigma, \tau \in K$. When $\sigma$ and $\tau$ are
  vertices this follows from the injectivity of $\Psi$ and $f$. When
  $\sigma$ and $\tau$ are edges this follows from the same identity
  for $\Psi$ and the fact that Step~4 guaranteed that $W_i \cap W_j =
  f(\sigma_i \cap \sigma_j)$ for $i,j \in [r], i \neq j$. The
  remaining case is when $\sigma=\sigma_i$ is an edge and $\tau$ a
  vertex. Then, by construction, $\tau \in \sigma_i$ if and only if
  $f(\tau) \in W_i$, and
  \[ \Phi(\sigma_i) \cap \Phi(\tau) = \Psi(W_i) \cap \Psi(f(\tau)) =
  \Psi(W_i \cap f(\tau)) = \left\{ \begin{array}{cl} \Psi(\emptyset) &
      \hbox{ if } f(\tau) \notin W_i\\ \Psi(f(\tau))& \hbox{ if }
      f(\tau) \in W_i\end{array}\right\} = \Phi(\sigma_i \cap \tau).\]

  It remains to check that for any simplex $\sigma \in K$, the support
  of $\gamma(\sigma)$ is contained in $\U{\Phi(\sigma)}$. When
  $\sigma=\{v\}$ is a vertex then $\gamma(\sigma) = \gamma'(f(v))$.
  Since $\gamma'$ is constrained by $(\F, \Psi)$, the support of
  $\gamma'(f(v))$ is contained in $\U{\Psi(f(v))} = \U{\Phi(v)}$, so
  the property holds. When $\sigma=\sigma_i$ is an edge,
  $\gamma(\sigma_i)$ is, by construction, a chain in $\U{\Psi(W_i)} =
  \U{\Phi(\sigma_i)}$ and the property also holds.
%
\end{proof}

%
%

\subsection{The induction}\label{s:induction}

Let $k \ge 2$, let $K$ be a simplicial complex of dimension $k$ and
assume that Proposition~\ref{p:ccm} holds for all simplicial complexes
of dimension $k-1$ or less. Let $\F$ be a finite family of subsets of
a topological space $\Rspace$ such that for any $\G \subsetneq \F$ and
any $0\le i \le k-1$, $\bigcap\G \neq \emptyset$ and
$\beta_i\pth{\cap\G} \le b$. Assuming that $\F$ contains sufficiently
many sets, we want to construct a nontrivial chain map $\gamma: C_*(K)
\to C_*(\Rspace)$ constrained by~$\F$.


\paragraph{Preliminary example.} When going from $k=0$ to $k=1$, the
first step (as described in Section~\ref{s:k=1}) is to start with a
constrained chain map $\gamma':C_*(\skel{0}{K}) \to C_*(\Rspace)$ and
observe that for some $1$-simplices
\noindent
\begin{minipage}{5cm}
  \begin{center}\includegraphics[width=5cm,keepaspectratio]{issue}\end{center}
\end{minipage}
\hfill
\begin{minipage}{10cm}
  $\{u,v\} \in K$ the chain $\gamma'(\partial\{u,v\})$ must already be
  a boundary. To see that this is not the case in general, consider
  the drawing of $\skelsim{1}{4}$ in an annulus depicted in the figure
  on the left.  Observe that for every triangle $\{i,j,k\} \in
  \skelsim{2}{4}$ the image, in this drawing, of $\partial\{i,j,k\}$
  is a cycle going around the hole of the annulus and is therefore not
  a boundary. So, if we start with a chain map $\gamma'$ corresponding
  to that drawing, we will not be able to extend it by ``filling'' any
  triangle directly. This is not a peculiar example, and a similar
  construction can easily be done with arbitrarily many vertices.
  Observe, though, that the cycle going from $1$ to $2$, then $4$,
  then $3$ and then back to $1$ \emph{is} a boundary; \hfill in other words,
  if we 
\end{minipage}

\smallskip
\noindent
replace, in the triangle $\partial\{1,2,3\}$, the edge from $2$ to $3$
by the concatenation of the edges from $2$ to $4$ and from $4$ to $3$,
we build, using a chain map of $\skelsim{1}{4}$ where no $2$-face can
be filled, a chain map of $\skelsim{2}{2}$ where the $2$-face can be
filled. We systematize this observation using the barycentric
subdivision of $K$.


\paragraph{Barycentric subdivision.} The idea behind the notion of
\emph{barycentric subdivision} is that the geometric realization of a
simplicial complex $K'$ can be subdivided by inserting a vertex at the
barycentre of every face, resulting in a new, finer, simplicial
complex, denoted $\sd K'$, that is still homeomorphic to $K'$. 
Formally, the vertices of $\sd K'$ consist of the
faces of $K'$, except for the empty face, and the faces of $\sd K'$
are the collections $\{\sigma_1, \dots, \sigma_\ell\}$ of faces of
$K'$ such that
\[ \emptyset \neq \sigma_1 \subsetneq \sigma_2 \subsetneq \cdots
\subsetneq \sigma_\ell.\]
In other words, the set of vertices of $\sd K'$ is $K'\setminus
\{\emptyset\}$ and the faces of $\sd K'$ are the chains of
$K'\setminus \{\emptyset\}$. For $\sigma \in K'$ we abuse the notation
and let $\sd \sigma$ denote the subdivision of $\sigma$ regarded as a
subcomplex of $\sd K'$, that is,
\[ \sd \sigma = \{ \{\sigma_1, \dots, \sigma_\ell\} \subseteq K' :
\emptyset \neq \sigma_1 \subsetneq \sigma_2 \subsetneq \cdots
\subsetneq \sigma_\ell \subseteq \sigma \}.\]
We will mostly manipulate barycentric subdivisions through the $\sd
\sigma$. For further reading on barycentric subdivisions we refer the
reader, for example, to~\cite[Section~1.7]{Matousek:BorsukUlam-2003}.

\paragraph{Overview of the construction of $\gamma$.} Let $s \in \N$
be some parameter depending on $K$ and to be determined later. To
construct $\gamma$ we will define three auxiliary chain maps
\[ C_*\pth{K^{(k-1)}} \quad \xrightarrow{\makebox[2em]{$\alpha$}}
\quad C_*\pth{\skel{k-1}{(\sd K)}} \quad
\xrightarrow{\makebox[3em]{$\beta_\sharp$}} \quad
C_*\pth{\skelsim{k-1}{s}} \xrightarrow{\makebox[2em]{$\gamma'$}} \quad
C_*(\Rspace)\]
As before, $\gamma'$ is a chain map from $C_*(\skelsim{k-1}{s})$
constrained by $\F$ and is obtained by applying the induction
hypothesis. Unlike in Section~\ref{s:k=1}, we do not inject the
vertices of $K$ into those of $\simplex{s}$ directly but proceed through
$\sd K$, the barycentric subdivision of $K$. We ``inject'' $K^{(k-1)}$ into
$\sd K^{(k-1)}$ by means of a chain map $\alpha$ (which will be the standard
chain map corresponding to a subdivision). We then construct an
injection $\beta$ of the vertices of $\sd K$ into the vertices of
$\simplex{s}$ which we extend linearly into a chain map
$\beta_\sharp$. The key idea is the following:
\begin{quote}
  The boundary of any $k$-simplex $\sigma$ of $K$ is mapped, under
  $\alpha$, to a sum of $k!$ boundaries of $k$-simplices of $\sd K$,
  all of which are mapped through $\beta_\sharp$ to chains with the
  same homology in some appropriate $\U{W_\sigma}$.
\end{quote}
Since $k!$ is even and we consider homology with coefficients in
$\Z_2$, it follows that $\gamma' \circ \beta_\sharp \circ \alpha
(\sigma)$ is a boundary in $\U{W_\sigma}$. We therefore construct
$\gamma$ as an extension of $\gamma' \circ \beta_\sharp \circ \alpha$.


\paragraph{Definition of $\gamma'$.}

Since $\skelsim{k-1}{s}$ has dimension $k-1$, the induction hypothesis
ensures that if the cardinality of $\F$ is large enough then there
exists a nontrivial chain map $\gamma':C_*(\skelsim{k-1}{s}) \to
C_*(\Rspace)$ constrained by $\F$. We denote by $\Psi$ a map such that
$\gamma'$ is constrained by $(\F,\Psi)$. Remark that $\Psi$ must be
monotone over $\skelsim{k-1}{s}$ as for any $\sigma \subseteq \tau \in
\skelsim{k-1}{s}$ we have $\Psi(\sigma) = \Psi(\sigma \cap \tau) =
\Psi(\sigma) \cap \Psi(\tau) \subseteq \Psi(\tau)$. It follows that
for any $\sigma \in \skelsim{k-1}{s}$ we have
\[ \Psi(\sigma) = \bigcup_{\tau \in \skelsim{k-1}{s}, \tau \subseteq
  \sigma} \Psi(\tau)\]
We use this identity to extend $\Psi$ to $\simplex{s}$, that is we
define:
\[ \forall A \subseteq V(\simplex{s}), \quad \Psi(A)=\bigcup_{\tau
  \in \skelsim{k-1}{s}, \tau \subseteq A}\Psi(\tau).\]
Remark that the extended map still commutes with the intersection:


\begin{lemma}\label{l:commute}
  For any $A, B \subseteq V(\simplex{s})$ we have $\Psi(A) \cap
  \Psi(B) = \Psi(A \cap B)$.
\end{lemma}
\begin{proof}
  For any $A, B \subseteq V(\simplex{s})$ we have
  \[\Psi(A)\cap \Psi(B) =
  \pth{\bigcup_{\sigma \in \skelsim{k-1}{s},\sigma\subseteq A}
    \Psi(\sigma)}\cap \pth{\bigcup_{\tau \in
      \skelsim{k-1}{s},\tau\subseteq B} \Psi(\tau)}\]
  Distributing the union over the intersections we get
  \[ \Psi(A)\cap \Psi(B) = \bigcup_{\sigma,\tau \in
      \skelsim{k-1}{s},\sigma\subseteq A,\tau\subseteq B}
    \Psi(\sigma)\cap \Psi(\tau)
  \]
  and as $\Psi(\sigma \cap \tau) = \Psi(\sigma) \cap \Psi(\tau)$ if
  $\sigma,\tau$ are simplices of $\skelsim{k-1}{s}$, this rewrites as
  \[
  \Psi(A)\cap \Psi(B) =\bigcup_{\sigma,\tau \in
    \skelsim{k-1}{s},\sigma\subseteq A,\tau\subseteq B} \Psi(\sigma
  \cap \tau).\]
  Finally, observing that
  \[ \{\sigma \cap \tau : \sigma,\tau \in
  \skelsim{k-1}{s},\sigma\subseteq A,\tau\subseteq B\} = \{\vartheta :
  \vartheta \in \skelsim{k-1}{s},\vartheta \subseteq A \cap B\}\]
  we get
  \[ \Psi(A)\cap \Psi(B) = \bigcup_{\vartheta \in \skelsim{k-1}{s},
    \vartheta \subseteq A \cap B} \Psi(\vartheta) = \Psi(A\cap B)\]
  which proves the desired identity.
\end{proof}


\paragraph{Definition of $\alpha$.} Now we define a chain map $\alpha:
C_*\pth{K^{(k-1)}} \to C_*\pth{\sd K^{(k-1)}}$ by first putting
\[ \alpha: \sigma \in K^{(k-1)} \mapsto
\mathop{\sum_{\tau \in \sd \sigma}}_{\dim \tau = \dim \sigma}
\tau,\]
and then extending that map linearly to $C_*\pth{K^{(k-1)}}$. See
Figure~\ref{f:alpha}. Remark that
$\alpha$ behaves nicely with respect to the differential:
\[ \alpha(\partial\sigma) = \mathop{\sum_{\tau \in \sd \sigma}}_{\dim
  \tau = \dim \sigma}
\partial\tau.\]
Note that the formula above makes sense and is valid even if $\sigma$ is a
$k$-simplex although we define $\alpha$ only up to dimension $k-1$.

%

\begin{figure}
  \begin{center}
    \includegraphics{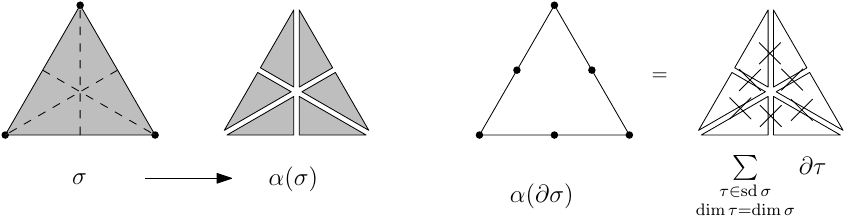}
    \caption{The map $\alpha$ applied to a simplex $\sigma$ (left) and to
    $\partial \sigma$ (right). Significant parts of the boundaries 
    $\partial \tau$ cancel out.}
    \label{f:alpha}
  \end{center}
\end{figure}

\paragraph{Definition of $\beta$.} We now construct the injection
$\beta: V(\sd K) \to V(\simplex{s})$ and, for constraining purposes,
an auxiliary function $\kappa$ associating with every $k$-dimensional
simplex of $K$ some simplex of $\simplex{s}$. We want these functions
to satisfy:
\begin{itemize}
\item[(P1)] For any simplex $\sigma \in K$, $\kappa(\sigma) \cap \ima
  \beta = \beta(V(\sd \sigma))$.
\item[(P2)] For any $k$-simplices $\sigma, \tau \in K$,
  $\kappa(\sigma) \cap \kappa(\tau) = \beta(V(\sd \sigma)) \cap
  \beta(V(\sd \tau))$.
\item[(P3)] For any $k$-simplex $\sigma \in K$, when $\tau$ ranges
  over all $k$-simplices of $\sd \sigma$, all chains $\gamma'\circ
  \beta_\sharp(\partial \tau)$ have support in
  $\U{\Psi(\kappa(\sigma))}$ and are in the same homology class in
  $H_{k-1}(\U{\Psi(\kappa(\sigma))})$.
\end{itemize}
The intuition behind these properties is that $\kappa(\sigma)$ should
augment $\beta(V(\sd \sigma))$ by ``dummy'' vertices (P1) in a way
that distinct simplices use disjoint sets of ``dummy'' vertices (P2).
Property~(P3), will allow building $\gamma$ over $k$-simplices as
explained in the preceding overview.

We start the construction of $\beta$ and $\kappa$ with a combinatorial
lemma. Let $\ell=2^{k+1}-1$ stand for the number of vertices of the
barycentric subdivision of a $k$-dimensional simplex, and set $m
=R_{k+1}(2^b,\ell)$.

\begin{claim}
\label{c:uniform-i}
For any integer $t$, if $s \ge R_m\pth{\binom{m}{\ell},t}$ then there
exist a set $T$ of $t$ vertices of $\simplex{s}$ and a set $Q^* \in
\binom{[m]}{\ell}$ such that $Q^*$ selects in any $M \in \binom{T}{m}$
a subset $L_M$ with the following property: when $\sigma$ ranges over
all $k$-simplices of $\simplex{s}$ with $\sigma \subseteq L_M$, all
chains $\gamma'(\partial \sigma)$ are in the same homology class in
$H_{k-1}\left(\U{\Psi(M)}\right)$.
\end{claim}
\begin{proof}
  Let $M$ be a subset of $m$ vertices of $\Delta_s$. Since $\gamma'$
  is constrained by $(\F,\Psi)$, for every $k$-simplex $\sigma
  \subseteq M$ the support of $\gamma'(\partial \sigma)$ is contained
  in $\U{\Psi( \partial \sigma)} \subseteq \U{\Psi(\sigma) }\subseteq \U{\Psi(M)}$. 
  We can therefore color
  the $(k+1)$-uniform hypergraph on $M$ by assigning to every
  hyperedge $\sigma$ the homology class of $\gamma'(\partial \sigma)$
  in $\U{\Psi(M)}$. Since $\beta_{k-1}\pth{\U{\Psi(M)}} \le b$, there
  are at most $2^b$ colors in this coloring. As $m=R_{k+1}(2^b,\ell)$,
  Ramsey's Theorem implies that there exists a subset $L \subset M$ of
  $\ell$ vertices inducing a monochromatic hypergraph. We let $Q_M$
  denote an element of $ \binom{[m]}{\ell}$ that selects such a subset
  $L$.

  It remains to find a subset $T$ of vertices of $\Delta_s$ so that
  all $m$-element subsets $M \subseteq T$ give rise to the same $Q_M$.
  This is done by another application of Ramsey's theorem to the
  $m$-uniform hypergraph on the vertices of $\Delta_s$ where each
  hyperedge $M$ is colored by the $\ell$-element subset $Q_M$. The
  subset $T$ can have size $t$ as soon as $s \ge
  R_m\pth{\binom{m}{\ell},t}$, which proves the statement.
\end{proof}

Now, back to the construction of $\beta$ and $\kappa$. We first want a
subset of $V(\simplex{s})$ with a ``uniform $\ell$-in-$m$ selection''
property of Claim~\ref{c:uniform-i} large enough so that we can inject
$V(\sd K)$ using Lemma~\ref{l:rescale}. We set:
\[ t= v(\sd K)+r(m-\ell) \quad \hbox{and} \quad s^* =
R_m\pth{\binom{m}{\ell},t},\]
and assume that $s \ge s^*$; since $s^*$ only depends on $b$ and $K$,
this merely requires that $\F$ is large enough, again as a function of
$b$ and $K$, so that $\gamma'$ still exists. We let $T$ and $Q^*$
denote the subset of $V(\simplex{s})$ and the element of
$\binom{[m]}{\ell}$ whose existence follows from applying
Claim~\ref{c:uniform-i}. Let $\sigma_1, \sigma_2, \ldots, \sigma_r$
denote the $k$-dimensional simplices of $K$. We apply
Lemma~\ref{l:rescale} with
\[ Y= V(\sd K), \quad Z = T, \quad A_i = V(\sd \sigma_i), \quad
q=\ell, \quad \hbox{and } w=m,\]
and obtain an injection $\pi: Y \to Z$ and $W_1, W_2, \ldots, W_r \in
\binom{Z}{m}$ such that (i) for every $i \le r$, $Q^*$ selects
$\pi(A_i)$ in $W_i$, and (ii) for any $i\neq j \le r$, $W_i \cap W_j =
\pi(A_i \cap A_j)$. This injection $\pi$ is our map $\beta$ and we put
$\kappa(\sigma_i) = W_i$. It is clear that Property~(P1) holds, and since 
\[ \kappa(\sigma_i) \cap \kappa(\sigma_j) = W_i \cap W_j = \pi(A_i
\cap A_j) = \beta(V(\sd \sigma_i) \cap V(\sd \sigma_j)) = \beta(V(\sd
\sigma_i)) \cap \beta(V(\sd \sigma_j)),\]
Property~(P2) also holds. The set $Q^*$ selects $\pi(A_i)$ in $W_i$
(Lemma~\ref{l:rescale}) so Claim~\ref{c:uniform-i} ensures that when
$\tau$ ranges over all $k$-simplices of $\simplex{s}$ with $\tau
\subseteq \pi(A_i)$, all chains $\gamma'(\partial \tau)$ have support
in $\U{\Psi(W_i)}$ and are in the same homology class in
$H_{k-1}\left(\U{\Psi(W_i)}\right)$. Substituting $\pi(A_i) = \beta(V(\sd
\sigma_i))$ and $W_i = \kappa(\sigma_i)$, we see that~(P3) holds.

\paragraph{Construction of $\gamma$.} Recall that we have the chain
maps\footnote{$\beta_\sharp$ is the chain map induced by $\beta$ restricted to
  chains of dimension
  at most $(k-1)$.}:
\[ C_*\pth{K^{(k-1)}} \quad \xrightarrow{\makebox[2em]{$\alpha$}}
\quad C_*\pth{\skel{k-1}{(\sd K)}} \quad
\xrightarrow{\makebox[3em]{$\beta_\sharp$}} \quad
C_*\pth{\skelsim{k-1}{s}} \xrightarrow{\makebox[2em]{$\gamma'$}} \quad
C_*(\Rspace).\]
We define $\gamma = \gamma' \circ \beta_\sharp \circ \alpha$ as a
chain map from $C_*\pth{K^{(k-1)}}$ to $C_*(\Rspace)$. Let $\sigma$ be a
$k$-dimensional simplex of $K$. From the definition of $\alpha$ we
have
\[\gamma\pth{\partial \sigma} = \mathop{\sum_{\tau \in \sd
    \sigma}}_{\dim \tau = \dim \sigma} \gamma^\prime \circ
\beta_\sharp(\partial \tau).\]
By property~(P3), all summands in the above chain have support in
$\U{\Psi(\kappa(\sigma))}$ and belong to the same homology class in
$H_{k-1}\left(\U{\Psi(\kappa(\sigma))}\right)$. There is an even number of
summands, namely $k!$ and we are using homology over $\Z_2$, so
$\gamma' \circ \beta_\sharp \circ \alpha (\partial \sigma)$ has
support in $\U{\Psi(\kappa(\sigma))}$ and is a boundary in
$\U{\Psi(\kappa(\sigma))}$.  We can therefore extend $\gamma$ into a
chain map from $C_*(K)$ to $C_*(\Rspace)$ in a way that for any $k$-simplex
$\sigma$ of $K$, the support of $\gamma(\sigma)$ is contained in
$\U{\Psi(\kappa(\sigma))}$.

\paragraph{Properties of $\gamma$.} First we verify that $\gamma$ is nontrivial. If $v$ is a vertex of $K$ then
$\sd v$ consists of a single simplex, also a vertex. The chain
$\alpha(v)$ is thus a single vertex of
$\sd K$, and $\beta_\sharp \circ \alpha(v)$ is still a single
vertex $\beta(\sd v)$. Since $\gamma'$ is nontrivial, the support of
$\gamma(v)$ is an odd number of points and therefore $\gamma$ is also
nontrivial. It remains to argue that $\gamma$ is constrained by
$(\F,\Phi)$ where:
\[ \Phi:\left\{\begin{array}{rcl} K & \to & 2^\F \\ \sigma & \mapsto &
    \left\{ \begin{array}{ll} \Psi(\beta(V(\sd\sigma))) & \hbox{if } \dim \sigma \le k-1 \\
        \Psi(\kappa(\sigma)) & \hbox{if } \dim \sigma =
        k\end{array}\right.\\\end{array}\right.\]
It is clear that $\Phi(\emptyset) = \Psi(\emptyset) = \emptyset$ by
definition of $\Psi$. Also, the construction of $\gamma$ immediately
ensures that for any $\sigma \in K$ the support of $\gamma(\sigma)$ is
contained in $\U{\Phi(\sigma)}$. To conclude the proof that $\gamma$
is constrained by $(\F,\Phi)$ and therefore the induction it only
remains to check that $\Phi$ commutes with the intersection:

\begin{claim}
  For any $\sigma, \tau \in K$, $\Phi(\sigma \cap \tau) = \Phi(\sigma)
  \cap \Phi(\tau)$.
\end{claim}
\begin{proof}
  The claim is obvious for $\sigma = \tau$, so from now on assume that this is not the case.
  First assume that $\sigma$ and $\tau$ have dimension at most $k-1$.
  Then,
  \[ \Phi(\sigma) \cap \Phi(\tau) = \Psi(\beta(V(\sd \sigma))) \cap
  \Psi(\beta(V(\sd \tau))) = \Psi(\beta(V(\sd \sigma)) \cap \beta(V(\sd
  \tau))),\]
  the last equality following from Lemma~\ref{l:commute}. Since the
  map $\beta$ on subsets of $V(\simplex{s})$ is induced by a map $\beta$
  on vertices of $\simplex{s}$ we have $\beta(V(\sd \sigma)) \cap
  \beta(V(\sd \tau)) = \beta(V(\sd \sigma) \cap V(\sd \tau))$.
  Moreover, by the definition of the barycentric subdivision we have
  $V(\sd \sigma) \cap V(\sd \tau) = V(\sd (\sigma \cap \tau))$. Thus,
  \[ \Psi(\beta(V(\sd \sigma)) \cap \beta(V(\sd \tau))) = \Psi
  (\beta(V(\sd(\sigma \cap \tau)))) = \Phi(\sigma \cap \tau),\]
  and the statement holds for simplices of dimension at most $k-1$.

\medskip

  Now assume that $\sigma$ and $\tau$ are both $k$-dimensional so
  that
  \[ \Phi(\sigma) \cap \Phi(\tau) = \Psi(\kappa(\sigma)) \cap
  \Psi(\kappa(\tau)) = \Psi(\kappa(\sigma) \cap \kappa(\tau)) =
  \Psi(\beta(V(\sd \sigma)) \cap \beta(V(\sd \tau))),\]
  the last identity following from Property~(P2) of the map $\kappa$.
  Again, from the definition of $\beta$ and the barycentric
  subdivision we have
  \[ \beta(V(\sd \sigma)) \cap \beta(V(\sd \tau)) = \beta(V(\sd(\sigma
  \cap \tau))).\]
  We thus obtain
  \[ \Phi(\sigma) \cap \Phi(\tau) = \Psi \circ \beta \circ V(\sd (\sigma \cap
  \tau)) = \Phi(\sigma \cap \tau),\]
  the last identity following from the definition of $\Phi$ on
  simplices of dimension at most $k-1$. The statement also holds for
  simplices of dimension $k$.

  \medskip

  Finally assume that $\sigma$ and $\tau$ are of dimension $k$ and at
  most $k-1$ respectively. Then, applying Lemma~\ref{l:commute} we have:
  \[ \Phi(\sigma) \cap \Phi(\tau) = \Psi(\kappa(\sigma)) \cap
  \Psi(\beta(V(\sd \tau))) = \Psi(\kappa(\sigma)\cap\beta(V(\sd
  \tau))).\]
  Note that $\beta(V(\sd \tau)) \subseteq \ima \beta$ and that, by
  property~(P1), $\kappa(\sigma) \cap \ima \beta = \beta(V(\sd
  \sigma))$. We thus have
  \[ \kappa(\sigma) \cap \beta(V(\sd \tau)) = \beta(V(\sd \sigma))
  \cap \beta(V(\sd \tau)) = \beta(V(\sd(\sigma \cap \tau))),\]
  the last equality following, again, from the definition of
  barycentric subdivision. As $\sigma \cap \tau$ has dimension at most
  $k-1$ we have
  \[ \Phi(\sigma) \cap \Phi(\tau) = \Psi(\beta(V(\sd(\sigma \cap
  \tau)))) = \Phi(\sigma \cap \tau)\]
  and the statement holds for the last case.
\end{proof}

\bibliographystyle{alpha}
\bibliography{hb}

\clearpage
\appendix

\end{document}